\documentclass{amsart}

\usepackage{stmaryrd}
\usepackage{amssymb, amsfonts, amsmath, amsthm,bm, amscd}
\usepackage[dvipsnames]{xcolor}
\usepackage{mathrsfs}
\usepackage{hyperref}
\usepackage{graphicx}
\usepackage{caption}
\usepackage{subcaption}
\usepackage{wrapfig}
\usepackage[margin=1.35 in]{geometry}
\usepackage{enumitem, multicol}
\usepackage{tikz}
\usetikzlibrary{cd}
\usepackage{nicefrac, bigints}
\usepackage{array}

\newcommand{\Addresses}{{
  \bigskip
  \footnotesize
  James Farre,  
    \textsc{Mathematisches Institut, Ruprecht-Karls Universit\"at Heidelberg}\par\nopagebreak
    \textit{E-mail address}:
  \texttt{jfarre@mathi.uni-heidelberg.de}
  }}

\newcommand{\para}[1]{\medskip\noindent\textbf{#1.}}

\theoremstyle{definition}
\newtheorem{theorem}{Theorem}[section]

\newtheorem{lemma}[theorem]{Lemma}

\newtheorem{corollary}[theorem]{Corollary}
\newtheorem{proposition}[theorem]{Proposition}
\newtheorem{remark}[theorem]{Remark}

\newtheorem{claim}[theorem]{Claim}

\newtheorem*{note}{Note}

\makeatletter
\DeclareRobustCommand{\cev}[1]{%
  {\mathpalette\do@cev{#1}}%
}
\newcommand{\do@cev}[2]{%
  \vbox{\offinterlineskip
    \sbox\z@{$\m@th#1 x$}%
    \ialign{##\cr
      \hidewidth\reflectbox{$\m@th#1\vec{}\mkern4mu$}\hidewidth\cr
      \noalign{\kern-\ht\z@}
      $\m@th#1#2$\cr
    }%
  }%
}
\makeatother

\newcommand{\ZZ}{\mathbb{Z}}
\newcommand{\RR}{\mathbb{R}}
\newcommand{\QQ}{\mathbb{Q}}
\newcommand{\CC}{\mathbb{C}}
\newcommand{\HH}{\mathbb{H}}

\newcommand{\inverse}{^{-1}}
\newcommand{\ddt}{\left. \frac{d}{dt}\right|_{t= 0}}
\newcommand{\ddtp}{\left. \frac{d}{dt}\right|_{t= 0^+}}

\newcommand{\cN}{\mathcal{N}}

\DeclareMathOperator{\Mod}{Mod}
\DeclareMathOperator{\Ad}{Ad}

\newcommand{\MF}{\mathcal{MF}}
\newcommand{\PML}{\mathcal{PML}}
\newcommand{\ML}{\mathcal{ML}}

\newcommand{\GL}{\mathcal{GL}}
\DeclareMathOperator{\SL}{SL}

\DeclareMathOperator{\Cos}{Cos}

\newcommand{\T}{\mathcal{T}}

\newcommand{\cH}{\mathcal{H}}
\newcommand{\cQ}{\mathcal{Q}}

\newcommand{\cR}{\mathcal{R}}

\newcommand{\PT}{\mathcal{PT}}
\newcommand{\PM}{\mathcal{PM}}
\newcommand{\PoT}{\mathcal{P}^1\mathcal{T}}
\newcommand{\PoM}{\mathcal{P}^1\mathcal{M}}
\newcommand{\QT}{\mathcal{QT}}
\newcommand{\QM}{\mathcal{QM}}

\DeclareMathOperator{\Sp}{\mathsf{Sp}}

\DeclareMathOperator{\Isom}{Isom}
\DeclareMathOperator{\PSL}{\mathsf{PSL}}
\DeclareMathOperator{\PSO}{\mathsf{PSO}}

\DeclareMathOperator{\inj}{inj}

\DeclareMathOperator{\Area}{Area}
\DeclareMathOperator{\Hom}{Hom}
\newcommand{\tlambda}{\widetilde{\lambda}}
\newcommand{\tX}{\widetilde{Z}}

\newcommand{\cO}{\mathcal{O}}

\DeclareMathOperator{\Th}{Th}
\DeclareMathOperator{\WP}{WP}

\DeclareMathOperator{\Eq}{Eq}
\DeclareMathOperator{\tr}{tr}

\newcommand{\Tx}{{T_z}}

\DeclareMathOperator{\diam}{diam}
\DeclareMathOperator{\stre}{stretch}

\begin{document}

\title[]{Hamiltonian flows for pseudo-Anosov mapping classes}

\author{James Farre}
\begin{abstract}
For a given pseudo-Anosov homeomorphism $\varphi$ of a closed surface $S$, the action of $\varphi$ on the Teichm\"uller space $\T(S)$ preserves the Weil-Petersson symplectic form.  We give explicit formulae for two invariant functions $\T(S)\to \RR$ whose symplectic gradients generate autonomous Hamiltonian flows that coincide with the action of $\varphi$ at time one.  We compute the Poisson bracket between these two functions.  This amounts to computing the variation of length of a H\"older cocyle on one lamination along a shear vector field defined by another.  For a measurably generic set of laminations, we prove that the variation of length is expressed as the cosine of the angle between the two laminations integrated against the product H\"older distribution, generalizing a result of Kerckhoff.  We also obtain rates of convergence for the supports of germs of differentiable paths of measured laminations in the Hausdorff metric on a hyperbolic surface, which may be of independent interest.  
\end{abstract}

\vspace*{-4em}
\maketitle
\vspace*{-2em}

\section{Introduction}

\subsection{Main results}
We are interested in \emph{Hamiltonian flows} associated to geometrically defined functions on the Teichm\"uller space $\T(S)$ of a  closed, oriented surface $S$ with negative Euler characteristic equipped with its Weil-Petersson symplectic form $\omega_{\WP}$.  Prominent examples are furnished by the classical Fenchel–Nielsen twist flow about a simple closed curve $\gamma$, generalized by Thurston to earthquake flows in measured laminations.  These functions are their flows are real analytic Hamiltonian flows with Hamiltonian potential the hyperbolic length  \cite{Wolpert:magic,Kerckhoff:NR}.  

The natural action of the mapping class group $\Mod(S)$ on $\T(S)$ preserves the symplectic form.  The Dehn twist $T_\gamma$ in $\gamma$ is thus a symplectomorphism of $\T(S)$, and the \emph{square} of the length function $-\frac12\ell_\gamma^2: \T(S)\to \RR$ is a Hamiltonian potential for $T_\gamma$; that is, the flow of the corresponding Hamiltonian vector field at time one is equal to the action of $T_\gamma$ on $\T(S)$.\footnote{The minus sign comes from the definition of the action: $T_\gamma.[f:S\to Z] = [f\circ T_\gamma\inverse : S\to Z]$.} Do other mapping classes admit Hamiltonian potentials?

Let $\varphi : S\to S$ be a pseudo-Anosov homeomorphism with projectively invariant measured laminations $\mu_+$ and $\mu_- \in \ML(S)$.  There is a real number $\Lambda>1$ such that $\varphi_*\mu_+ = \Lambda\mu_+$ and $\varphi_*\mu_- = \Lambda\inverse\mu_-$; the leaves of $\mu_+$ are stretched, while the leaves of $\mu_-$ are contracted under $\varphi$.
Although measured laminations space does not have a natural smooth structure, if $\mu_+$ (hence $\mu_-$) is maximal, the PL tangent space $T_{\mu_+}\ML(S)$ can be identified with a symplectic vector space of \emph{transverse H\"older distributions} $\cH(\mu_+)$, which is modeled on the weight space $W(\tau)$ of a train track $\tau$ that carries $\mu_+$, equipped with its \emph{Thurston symplectic form} $\omega_{\Th}$ \cite{Bon_GLTHB,Bon_THDGL}.

For a given $\alpha\in \cH(\mu_+)$ and small enough $s>0$, $\mu_++ s\alpha$ assigns positive numbers to the branches of $\tau$ satisfying the switch conditions.  The corresponding path of measured laminations converging in measure to $\mu_+$ represents the tangent direction $\alpha$.  The  \emph{length} $\ell_\alpha(Z)$ of $\alpha$ on a hyperbolic surface $Z\in \T(S)$ computes the derivative of the hyperbolic length of the measures $\mu_++s\alpha$ on $Z$ at time $s=0$; see \cite[Corollary 25]{Bon_GLTHB} or \S\ref{sec:background_shear}, below.

The differential of $\varphi$ induces a self map of the tangent space $T_{[\mu_+]}\mathcal {PML}(S) \cong \cH(\mu_+)/\langle \mu_+ \rangle$, which lifts to a linear symplectomorphism $\cH(\mu_+) \to \cH(\mu_+)$. 
Our first result gives a formula for a Hamiltonian potential for the action of $\varphi$ on $\T(S)$.  
We first assume that $\mu_+$ is maximal and that the linear action of $\varphi$ on $\cH(\mu_+)$ is diagonalizable over $\RR$ with positive spectrum.  
We give a detailed description of the linear action of $\varphi$ and lift all of the restrictions imposed above in \S\ref{sec:flow}. 

\begin{theorem}\label{thm:pA_Hamiltonian} 
The action of $\varphi\inverse$ on $\T(S)$ is the time one flow of the $\omega_{\WP}$-symplectic gradient of the function 
\begin{equation}
F^{\mu_+}_{\varphi\inverse}(Z) = \sum_{i =1}^{3g-3} \log(\Lambda_i) ~  \ell_{ \alpha_i}(Z)\cdot\ell_{ \beta_i}(Z),
\end{equation}
where $\Lambda = \Lambda_1\ge ... \ge \Lambda_{3g-3} \ge \Lambda_{3g-3}\inverse \ge ... \ge\Lambda_1\inverse$ are eigenvalues of the linear action of $\varphi$ on $\cH(\mu_+)$ with corresponding symplectic basis of expanding and contracting eigenvectors $\mu_+ = \alpha_1, ..., \alpha_{3g-3}$ and $\beta_1, ..., \beta_{3g-3}\in \cH(\mu_+)$.

Moreover, $F^{\mu_+}_{\varphi\inverse}$ vanishes on the unique $\varphi\inverse$-invariant stretch line forward directed by $\mu_+$, which is a Hamiltonian flow line parameterized proportionally to directed arclength with respect to Thurston's asymmetric Lipschitz metric on $\T(S)$.  
\end{theorem}

Specializing the discussion to the other invariant lamination $\mu_-$, Theorem \ref{thm:pA_Hamiltonian} produces a function 
\begin{equation}
F^{\mu_-}_\varphi(Z) = -\sum_{i =1}^{3g-3} \log(\Lambda_i) ~  \ell_{\ast\inverse\alpha_i}(Z)\cdot\ell_{ \ast\inverse\beta_i}(Z),
\end{equation}
whose corresponding Hamiltonian flow at time one is  equal to the action of $\varphi$.
There is a duality map (see \S\ref{sec:duality}) $\ast: \cH(\mu_-) \to \cH(\mu_+)$ so that $\ast\inverse\alpha_1, ..., \ast\inverse\alpha_{3g-3}$,  $\mu_- = \ast\inverse\beta_1, ..., \ast\inverse\beta_{3g-3}$ is a symplectic basis of contracting and expanding eigenvectors for the linear action of $\varphi\inverse$ on $\cH(\mu_-)$. Moreover, $F^{\mu_-}_\varphi$ vanishes on the unique $\varphi$-invariant stretch line forward directed by $\mu_-$. 

While the flows generated by $-F^{\mu_+}_{\varphi\inverse}$ and $F^{\mu_-}_\varphi$ coincide at time one, in general these functions and their flows are different. Indeed, Thurston stretch lines are not usually symmetric, even up to reparameterization.  

One way to quantify the difference between the flows generated by these two functions is to compute the \emph{Poisson bracket}
\[\{-F^{\mu_+}_{\varphi\inverse}, F^{\mu_-}_\varphi\} = \omega_{\WP}\left(-X_{F^{\mu_+}_{\varphi\inverse}}, X_{F^{\mu_-}_\varphi}\right) = -dF^{\mu_+}_{\varphi\inverse}  X_{F^{\mu_-}_\varphi},\]
measuring the change of one function along the flow of the other.  The Poisson bracket is antisymmetric, satisfies a Leibniz rule and a Jacobi identity, hence  endows $C^\infty (\T(S))$ with the structure of a Poisson algebra.  Using the Leibniz rule and linearity, the computation  of $\{-F^{\mu_+}_{\varphi\inverse}, F^{\mu_-}_\varphi\} $ (Corollary \ref{cor:Poisson_pA}) is reduced to the computation of $\{ \ell_{k_i}, \ell_{K_j}\}$ where $k_i \in \{\alpha_i, \beta_i\}$ and $K_j\in \{\ast\inverse\alpha_j, \ast\inverse\beta_j\}$.  

Let $\kappa$ be a partition of $4g-4$, and let $\cQ$ be a component of a stratum $\QM(\kappa)$ of the moduli space of holomorphic quadratic differentials.
There is a corresponding $\cQ$-Thurston measure $\mu_{\Th}^\cQ$ on $\ML(S)$, which is mutually singular with respect to the Thurston measure in the class of Lebesgue (unless $\cQ$ is the principal stratum, in which case $\mu_{\Th}^\cQ$ is the usual Thurston measure).  The $\mu_{\Th}^\cQ$-typical point is the horizontal measured foliations without horizontal saddle connections for differentials in $\cQ$ and can (usually) be described by topological invariants; see \S\ref{sec:stretch_flow}.
\begin{theorem}\label{thm:cosine_intro}

For $\mu_{\Th}^\cQ$-almost every measured geodesic lamination $\mu_1 \in \ML(S)$ and chain recurrent completion $\lambda_1$, for every chain recurrent geodesic lamination $\lambda_2$ meeting $\lambda_1$ transversally,  for all $\alpha_1\in \cH(\lambda_1)$ representing a tangent direction to $\mu_1$ in $\ML(S)$,  for all $\alpha_2\in \cH(\lambda_2)$,  and $Z\in \T(S)$, we have   \[\{\ell_{\alpha_1}, \ell_{\alpha_2}\} (Z) = \iint_Z \cos(\lambda_1, \lambda_2) ~d\alpha_1 d\alpha_2,\] 
where  $\cos(\lambda_1,\lambda_2): Z\to [-1,1]$ is the function supported on the transverse intersection $\lambda_1\pitchfork\lambda_2$ measuring the cosine of the angle made from leaves of $\lambda_1$ to leaves of $\lambda_2$.
\end{theorem}

We denote by $\Cos(\alpha_1, \alpha_2): \T(S)\to \RR$ the integral in Theorem \ref{thm:cosine_intro}, and emphasize that the integral is \emph{signed}, i.e., $\alpha_i$ are not measures, but only finitely additive signed measures (transverse H\"older distributions), and part of the proof of Theorem \ref{thm:cosine_intro} involves making sure that this integral makes sense.  In Section \ref{sec:int}, we prove that the integral is uniformly well approximated by certain Riemann sums defined by geometric train track splitting sequences (Proposition \ref{prop:R_sums_prod}) under suitable assumptions about the integrand.  

The $\mu_{\Th}^\cQ$ full measure set of measured geodesic laminations for which Theorem \ref{thm:cosine_intro} holds are those for which any quadratic differential with horizontal foliation equivalent to $\mu_1$ recurs to a compact set in $\QM(S)$ in backwards time under the Teichm\"uller geodesic flow.  Ergodicity for the Masur-Veech measure on the unit area locus of a component of a stratum implies that theses recurrent laminations indeed have full measure.
Evidently, the invariant Teichm\"uller geodesic axis for $\varphi$ satisfies this property for our periodic laminations $\mu_+$ and $\mu_-$.

\begin{remark}
There is an asymmetry in the statement of the Theorem \ref{thm:cosine_intro}.  Namely, we have made a strong assumption about the dynamical properties of $\lambda_1$ and almost no assumptions about $\lambda_2$.  However, anti-symmetry of the Poisson bracket  tells us that the cosine formula admits descriptions as a derivative where the roles of $\lambda_1$ and $\lambda_2$ are reversed, i.e. 
\[\{\ell_{\alpha_1}, \ell_{\alpha_2}\}(Z) = d\ell_{\alpha_1}X_{\alpha_2}(Z) = - d\ell_{\alpha_2}X_{\alpha_1}(Z),\] where $X_{\alpha_i}$ is the \emph{shear vector field} on $\T(S)$ symplectically dual to $\ell_{\alpha_i}$ (Lemma \ref{lem:length_dual}).  This suggests that perhaps the conclusion of the theorem should hold without restriction on $\lambda_1$.  
\end{remark}

Our Theorem \ref{thm:cosine_intro} generalizes a result of 
Kerckhoff, who expressed the variation of length of a measured lamination $\mu$ along the earthquake defined by another $\nu$ as $\Cos(\mu,\nu)$.
He proved \cite{Kerckhoff:NR, Kerckhoff:analytic} that length functions  are analytic and convex along earthquake paths.
Using different methods, Wolpert computed the first and second variation of length functions for curves under twist/shear deformations in other curves \cite{Wolpert:symplectic}.  Goldman produced generalizations of these results for the smooth points of $\Hom(\pi_1(S), G)\sslash G$, where $G$ is a connected Lie group with a non-degenerate $\Ad$-invariant bilinear form on its Lie algebra \cite{Goldman:curves}.  
There are also analogous results \cite{EM} for \emph{quakebends} in directions $a\mu+ib\mu$, $a, b\in \RR$ and $\mu\in \ML(S)$ for $\PSL_2\CC$ valued surface group representations. 
A special case of Theorem \ref{thm:cosine_intro} (Proposition \ref{prop:cosine}) was proved in \cite{G:derivatives}, where also some higher derivatives of the lengths of laminations in shearing vector fields were computed.  It is known that length functions of laminations are convex along the trajectories of shearing paths \cite{BBFS, Theret:convex}.

Let us  point out that Theorem \ref{thm:cosine_intro} actually computes a \emph{second} derivative of the length function $\ell: \T(S)\times \ML(S) \to \RR_{>0}$ in the sense that 
\[d\ell_{\alpha_1}X_{\alpha_2}(Z) = \ddtp\ell_{\lambda_1 + t\alpha_1}X_{\alpha_2} (Z), \]
so part of the proof involves exchanging certain limits and justifying the swap.  Care is required here, as the domain of this function does not even have a natural $C^1$ structure.
This essentially boils down to making sure that all estimates we give in the paper (especially \S\S \ref{sec:MH_convergence} and \ref{sec:int}) are uniform over compact subsets of $\T(S)$.  
The main technical ingredient needed to obtain Theorem \ref{thm:cosine_intro} is provided in \S\ref{sec:MH_convergence}, where we we give rates of  Hausdorff convergence for the supports of linear paths of measured laminations along recurrence times (Theorem \ref{thm:measure_H_close}) on a hyperbolic metric $Z\in \T(S)$.  This establishes a quantitative relationship between the measure topology and the Hausdorff topology on (supports of) measured laminations, which may be of independent interest.  

Having a geometric interpretation of these mixed partial derivatives may be useful for making geometric arguments in shear coordinates, as they appear in the differential of the change-of-shear-coordinates map  $\Sigma_{\lambda_2\lambda_1}: \cH^+(\lambda_1)\to \cH^+(\lambda_2)$, taking $\sigma_{\lambda_1}(Z)$ to $\sigma_{\lambda_2}(Z)$, for $Z\in \T(S)$.  We assume here  that $\lambda_i$ are maximal geodesic laminations satisfying the hypotheses of Theorem \ref{thm:cosine_intro}.
We give the Jacobian matrix for the derivative in terms of symplectic bases $x_1, ..., x_{3g-3}, y_1, ..., y_{3g-3}$ for $\cH(\lambda_1)$ and $z_1, ..., z_{3g-3}, w_1, ..., w_{3g-3}$ for $\cH(\lambda_2)$.

\begin{corollary}\label{cor:derivative_intro}
The derivative of $\Sigma_{\lambda_2\lambda_1} : \cH^+(\lambda_1)\to \cH^+(\lambda_2)$ at a point $\sigma_{\lambda_1}(Z)$ is given by
\[(\Sigma_{\lambda_2\lambda_1})_* = 
\begin{pmatrix}
-\{\ell_{w_i}, \ell_{x_j}\} & - \{\ell_{w_i},\ell_{y_j}\} \\
 \{\ell_{z_i},\ell_{x_j}\}  &  \{\ell_{z_i}, \ell_{y_j}\}
\end{pmatrix}
=
\begin{pmatrix}
\Cos(x_j, w_i) &\Cos(y_j,w_i) \\
\Cos(z_i, x_j) & \Cos(z_i, y_j)
\end{pmatrix}.\]
with inverse
\[(\Sigma_{\lambda_1\lambda_2})_* = 
\begin{pmatrix}
-\{\ell_{y_i}, \ell_{z_j}\} & - \{\ell_{y_i},\ell_{w_j}\} \\
\{ \ell_{x_i},\ell_{z_j}\} & \{\ell_{x_i}, \ell_{w_j}\}
\end{pmatrix}
= 
\begin{pmatrix}
\Cos(z_j, y_i) &\Cos(w_j,y_i) \\
\Cos(x_i, z_j) & \Cos(x_i, w_j)
\end{pmatrix}\]
All functions are evaluated at $Z$, and   $\{ \ell_{x_i},\ell_{z_j}\} $ represents the $3g-3$ square matrix $\left( \{ \ell_{x_i},\ell_{z_j}\} \right) _{i, j}$, etc.  
\end{corollary}

The following expresses the variation of length of a measured lamination along the stretch flow of another.
\begin{corollary}\label{cor:stretch_variation}
Let $\mu \in \ML(S)$ be maximal and generic in the sense of Theorem \ref{thm:cosine_intro}, let $Z\in \T(S)$, and let $\lambda$ be a chain recurrent geodesic lamination different from the support of $\mu$.  For any H\"older cocycle $\alpha\in \cH(\lambda)$,  we have  
\[d\log\ell_\alpha X_\mu^{\stre}(Z) = \frac{1}{\ell_\alpha(Z)}\iint_{Z} \cos(\lambda,\mu) ~ d\alpha d\sigma_\mu(Z),\]
where $X_\mu^{\stre}$ is the stretch vector field directed by $\mu$, and $\sigma_\mu(Z)\in \cH^+(\mu)$ measures the \emph{shear} of $Z$ along $\mu$ (see Theorem \ref{thm:coordinates}).  If $\alpha$ is a transverse measure, then this quantity is strictly bounded above by $1$. 
\end{corollary}

\subsection{Motivation and questions}
It seems natural to ask for geometrically meaningful functions on $\T(S)$ that induce the action of arbitrary mapping classes, as properties of the function correspond to properties of the flow.  For example, level sets are preserved.  Using only homogeneous quadratic polynomials in length functions of curves and H\"older distributions on geodesic laminations, it should be possible to obtain functions whose Hamiltonian flow generates the action of a suitable power of any mapping class at time one, although this technique yields very little or no information about finite order mapping classes.  

It seems like an interesting problem to understand the flows associated to products of length functions for intersecting curves or laminations, or even the flow associated to the length functions of curves with self intersections.
The product $\ell_{\mu_+} \ell_{\mu_-}$ is invariant under $\varphi$, and seems like an especially interesting example. Using Corollary \ref{cor:derivative_intro}, the corresponding Hamiltonian vector field $\ell_{\mu_+} X_{\mu_-} + \ell_{\mu_-} X_{\mu_+}$ can be expressed in terms of either the shearing vector fields associated to $\mu_+$ or to $\mu_-$.  Unfortunately, such expressions (in the notation of Theorem \ref{thm:pA_Hamiltonian})
\[\ell_{\mu_+} X_{\mu_-} + \ell_{\mu_-} X_{\mu_+}= \ell_{\mu_+} X_{\mu_-} +\ell_{\mu_-} \left( \sum_{i = 1}^{3g-3} \Cos(\alpha_i, \mu_-) X_{\alpha_i } + \Cos(\mu_-, \beta_i) X_{\beta_i}\right) \]
 seem difficult to extract meaning from, except perhaps asymptotically. 

The author was originally motivated by questions involving the action of a pseudo-Anosov mapping class on the smooth points of the variety  $\Hom(\pi_1(S), \PSL_2\CC)\sslash\PSL_2\CC$.  It is known that there are two interesting hyperbolic fixed points of this action on the boundary of the discrete and faithful locus coming from the cyclic cover associated to  the fiber subgroup of the hyperbolic mapping tori of $M_\varphi$ and $M_{\varphi\inverse}$  \cite{Kapovich:fp_hyperbolic}; see also \cite[Chapter 3]{McMullen:book}.  

Due to work of Bonahon \cite{Bon_SPB} that there are holomorphic shear coordinates on the locus of characters where a maximal lamination is \emph{realized}, using the complex valued Goldman symplectic form, Theorem \ref{thm:cosine_intro} has a straightforward generalization to shear-bend deformations.  
However, the realizable locus of characters for $\mu_+$ is a complex torus and the action of $\varphi$ is homotopically non-trivial, so the flow associated to any Hamiltonian function defined only on  the realizable locus cannot induce the action of $\varphi$ at time $1$.
The author does not know if inclusion of the realizable locus into the character variety is null-homotopic.  
So, it might still be possible to find a Hamiltonian function to generate the action on some appropriate open invariant set in the character variety if the corresponding flow does not preserve the realizable locus.

\subsection{Outline of the paper}
\S \ref{sec:background_shear} is dedicated to preliminaries on geodesic laminations and shear coordinates. 
In \S\ref{sec:flow}, we discuss the linear action of $\varphi$ and prove a more general version of Theorem \ref{thm:pA_Hamiltonian}, which becomes an exercise in symplectic linear algebra in shear coordinates.    
 In \S \ref{sec:geom_tt}, we discuss the geometry of horocyclically foliated train track neighborhoods of geodesic laminations.
 Then in \S\ref{sec:stretch_flow}, we review some recent work of Calderon-Farre and discuss the dynamics of the (generalized) stretch flow on strata of the moduli space of surface-lamination pairs. 
 We then use these dynamical results in \S\ref{sec:MH_convergence} to give a quantitative rate of convergence of germs of paths of measured laminations in the Hausdorff metric on a hyperbolic surface for recurrent measured laminations.  This is the main technical ingredient for the proof of Theorem \ref{thm:cosine_intro}.
 In \S\ref{sec:int}, we discuss H\"older geodesic currents and show that the integral appearing in Theorem \ref{thm:cosine_intro} can be approximated uniformly well by certain Riemann sums, defined in terms of (geometric) train track splitting sequences.  We also recall the shearing cocycle associated to a transverse H\"older cocyle $\alpha$ which generates the flow of the shear vector field $X_\alpha$.  
 The proof of Theorem \ref{thm:cosine_intro} is carried out in Section \ref{sec:cosine}.  We then prove the remaining corollaries from the introduction.

\subsection*{Acknowledgments} 
The author would like to thank the anonymous referee for comments that helped improve this manuscript, Yair Minsky, Brice Loustau, Aaron Calderon and Beatrice Pozzetti for enlightening conversations related to this work, as well as Gabriele Benedetti, Francis Bonahon, Nguyen-Thi Dang, Arnaud Maret, and Max Riestenburg for additional helpful discussion.  
The author gratefully acknowledges the support of DFG grant 427903332 (Emmy Noether), NSF grant DSM-2005328. Additionally, the work was funded by the Deutsche Forschungsgemeinschaft (DFG, German Research Foundation) – Project-ID 281071066 – TRR 191.

\setcounter{tocdepth}{1}
\tableofcontents

\section{Preliminaries on shear coordinates} \label{sec:background_shear}

In this section, we recall some basic facts about geodesic laminations on hyperbolic surfaces.
The discussion culminates in Theorem \ref{thm:coordinates}, in which we list some of the remarkable properties of \emph{shear coordinates} for Teichm\"uller space.
In subsequent sections, we give more context on some of the more technical aspects of the theory as we require them.

\subsection{Geodesic laminations}
We let $S$ be a closed, oriented surface of genus at least two so that $S$ admits a negatively curved Riemannian metric $m_{\text{aux}}$.
A \emph{geodesic lamination} $\lambda$ is a non-empty closed subset of $S$ that is a disjoint union of complete simple $m_{\text{aux}}$-geodesics called its \emph{leaves}.
The space $\mathcal {GL}(S)$ of geodesic laminations with its Hausdorff topology is compact and does not depend on $m_{\text{aux}}$.

A \emph{(simple) multi-curve} is a collection of pairwise disjoint simple closed geodesics.
We call a geodesic lamination \emph{chain recurrent} if it is approximated by multi-curves.
Denote by $\mathcal {GL}_0(S)$ the subspace of chain recurrent geodesic laminations.  

A lamination is \emph{connected} if it is connected as a topological space and \emph{minimal} if it has no non-trivial sublaminations.
Alternatively, a lamination is minimal if each of its leaves is dense.
Any $\lambda \in \mathcal {GL}(S)$ can be decomposed uniquely as a union of minimal sublaminations $\lambda_1, ..., \lambda_m$ and isolated leaves $\ell_1, ..., \ell_k$ that accumulate/spiral onto the minimal components.
We have bounds $0\le k\le 6g-6$ and $1\le m\le 3g-3$. 

The $m_{\text{aux}}$-geodesic completion of $S\setminus \lambda$ is a (possibly non-compact, disconnected) negatively curved surface with totally geodesic boundary and area equal to that of $m_{\text{aux}}$; we call this surface the \emph{cut surface} or the surface obtained by \emph{cutting open along $\lambda$}.
A lamination is \emph{filling} if the cut surface is a union of ideal polygons.
We let $\mathcal {GL}_\Delta(S)$ denote the space of minimal and filling geodesic laminations; the cut surface is a union of ideal polygons.
A lamination is \emph{maximal} if it is not a proper sublamination of any other geodesic lamination (the cut surface is a union of $4g-4$ ideal triangles).
One can always make a choice of adding finitely many isolated leaves that spiral onto the components of $\lambda$ to form a \emph{completion} $\lambda'$ that  is maximal.  If $\lambda$ is chain recurrent, we can find a completion $\lambda'$ that is also chain recurrent.

An extremely useful property of geodesic laminations is that their tangent line fields are \emph{universally} Lipschitz continuous.
Let $\hat d$ be a left invariant metric on $\PSL_2\RR \cong T^1\HH^2$ that is right invariant under $\PSO(2)$ and induces the hyperbolic metric on $\HH^2$. 

\begin{lemma}[{\cite[Lemma 1.1]{Kerckhoff:NR} or \cite[Lemma 5.2.6]{CEG}}]\label{lem:lipschitz}
There is a universal constant $c$ such that if $x$ and $x'\in \HH^2$ are at most $\epsilon<1$ apart and  $l$ and $l'$ are disjoint complete geodesics passing through $x$ and $x'$, respectively, then $\hat d (x_l, x'_{l'}) \le c \epsilon$, where $x_l$ and $x'_{l'}\in T^1\HH^2$ are tangent to $l$ and $l'$ at $x$ and $x'$ respectively.  
\end{lemma}

Geodesic laminations were introduced by Bill Thurston in \cite{Thurston:notes, Thurston:bulletin} and have become an important tool in various problems in Teichm\"uller theory, low dimensional geometry topology and dynamics. A comprehensive introduction to the structure theory for geodesic laminations can be found in \cite{CB}; see also \cite{CEG}.

\subsection{Measured laminations}
Let $\lambda\in \mathcal {GL}_0(S)$ be minimal.
If $\lambda$ is not a closed curve, then its intersection with any transversal $k$, i.e., a $C^1$ simple arc meeting $\lambda$ transversally, is a cantor set of Hausdorff dimension $0$ \cite{BS}.
We can associate a dynamical system on $k\cap \lambda$ by giving a transversal $k$ an orientation, which then gives a local orientation to the leaves of $\lambda$.
Following the leaves of $\lambda$ induces a homeomorphic first return map $P : k\cap \lambda \to k\cap \lambda$.
It can be shown that there are at most $3g-3$ $P$-invariant ergodic probability measures on $k\cap \lambda$ (see, e.g., \cite{Veech:IET}).
If $k'$ is a transversal isotopic to $k'$ through arcs transverse to $\lambda$, then the corresponding dynamical systems are conjugated by the isotopy, and the invariant  measures are in bijective correspondence.

A \emph{transverse measure} $\mu$ supported on $\lambda$ is the assignment of a positive Borel measure $\mu(k)$ to each transversal $k$ supported on $k\cap \lambda$ that is invariant under transverse isotopy and natural with respect to inclusion.  
We sometimes conflate the Borel measure $\mu(k)$ with its integral against the constant function $1$.   
A geodesic lamination $\lambda$ equipped with a transverse measure $\mu$ is a \emph{measured lamination}.  

\begin{note}
We will reserve the notation $\lambda \in \mathcal {GL}(S)$ for geodesic laminations without any additional structure and use the letter $\mu$ to refer to a measured geodesic lamination.  Abusing notation, we will also use the letter $\mu$ to refer to the support of $\mu$, if it should not cause any confusion.
\end{note}

The weak-$*$ topology for continuous, flow invariant functions on $T^1S$ makes the space $\ML(S)$ of measured laminations on $S$ a piecewise-integer-linear (PIL) cone manifold homeomorphic to $\mathbb{R}^{6g-6}\setminus \{0\}$ \cite{Thurston:notes, Thurston:bulletin} (see also \cite{CB, FLP, PennerHarer} for expositions).
The Dirac masses on multi-curves, i.e.,  simple multi-curves with the transverse measure assigning positive integer weighted Dirac masses to transverse intersections,  form a lattice of integer points; the positive cone over the set of multi-curves is dense in $\ML(S)$.

\subsection{Train tracks}\label{sec:tt}
Thurston invented \emph{train tracks} as an extremely useful tool for understanding both the geometry of and dynamical systems associated to  geodesic laminations.
We refer the reader to \cite{PennerHarer}, \cite{Penner:pA}, and \cite{Thurston:notes}.

A \emph{train track} $\tau\subset S$ is an embedded graph in $S$ with a $C^1$ structure at the vertices satisfying some additional geometrical and topological constraints.
The train tracks that are relevant to our discussion can all be obtained by looking at a suitable small neighborhood of a chain recurrent geodesic lamination on a hyperbolic surface; we formalize this later on in \S\ref{sec:geom_tt}.
The edges are called \emph{branches} and are denoted by $b(\tau)$.
The vertices are called \emph{switches} and are denoted by $s(\tau)$.
A train track $\tau$ is called \emph{generic} or \emph{trivalent} if all switches are trivalent.
A $C^1$ immersion of a ((bi)-infinite) interval into $\tau$ is called a \emph{train path}.
We usually consider two train tracks as identical if they are $C^1$-isotopic.

We say that a train track $\tau$ \emph{carries} a geodesic lamination $\lambda$ and write $\lambda\prec \tau$ if there is a $C^1$ map $S\to S$ (called a \emph{carrying map}) homotopic to the identity that takes every leaf of $\lambda$ to a train path in $\tau$.
If $\mu$ is a transverse measure with  support contained in $\lambda$, then to each branch $b$ of $\tau$ we can associate the $\mu$-measure of the fiber of the carrying map at any point $p$ of $b$.  
We identify $w_\mu$ with a vector in $\RR_{\ge0}^{b(\tau)}$.
The invariance property of the transverse measure implies that this vector satisfies the \emph{switch conditions}.
That is, the sum of the weights on branches coming in to $s$ is equal to the sum of the weights on branches going out.

Let $W(\tau) \subset \RR^{b(\tau)}$ be the linear subspace (called the \emph{weight space}) cut out by all the switch conditions.
Let $W^+(\tau) = W(\tau) \cap \RR_{\ge 0 }^{b(\tau)}$ denote the non-negative cone.  
If $W^+(\tau)$ contains any strictly positive vectors, i.e., a weight vector giving positive mass to every branch of $\tau$, then  $\tau$ is \emph{recurrent}.

If $\tau$ can be embedded in a hyperbolic surface with long branches and small geodesic curvature, then we call $\tau$ \emph{transversely recurrent}.\footnote{see \cite[Section 1.4]{PennerHarer} for a more precise formulation.  In \S\ref{sec:geom_tt}, we construct train tracks as quotients of certain neighborhoods of geodesic laminations on a hyperbolic surface, and so they are  transversely recurrent.}
To a recurrent and transversely recurrent train track $\tau$, and to any $w\in W^+(\tau)\setminus \{0\}$, we can construct a geodesic lamination $\lambda_w$ that is carried by $\tau$ and equipped with a transverse measure $\mu_w$ that corresponds to $w$ via the carrying map.

Thus to a bi-recurrent (recurrent and transversely recurrent) train track $\tau$, there is a map
\[W^+(\tau) \to \ML(S).\]
If such a train track is \emph{maximal}, i.e., every component of $S\setminus \tau$ is a triangle, then restricted to the strictly positive weights $U(\tau)\subset W^+(\tau)$, this map is a local homeomorphism.
We often abuse notation and identify the set of positive weights $U(\tau)$ for a bi-recurrent track $\tau$ with its image in $\ML(S)$.

Say that $\tau'$ is \emph{carried} by $\tau$ and write $\tau'\prec \tau$ if there is a $C^1$ carrying map $h:S\to S$ homotopic to the identity mapping all train paths in $\tau'$ to train paths in $\tau$.   
A carrying map induces a linear map $A_h: \RR^{b(\tau')} \to \RR^{b(\tau)}$ as follows:
Enumerate the branches $b(\tau) = \{b_1, ..., b_k\}$ and $b(\tau') = \{b_1', .., b_\ell'\}$, choose points $p_i$ in the interior of $b_i$, and define
\begin{equation}\label{eqn:incidence}
M_h = \left( \# h\inverse(p_i) \cap b_j \right).
\end{equation}
Then $M_h$ is a non-negative integer matrix that restricts to a map of weight spaces \[A_h: W(\tau') \to W(\tau)\] taking the non-negative cone in $W(\tau')$ into the non-negative cone in $W(\tau)$.
While $M_h$ depends on the specific carrying map and the points $p_i$, $A_h$ does not \cite[Section 2.1]{PennerHarer}.

There is a finite set of maximal bi-recurrent train tracks $\{\tau_i\}$ such that $\cup_i U(\tau_i) = \ML(S)$ and the transition functions are piecewise restrictions of integer linear isomorphisms on the overlaps in these coordinate chats.
Thus train track coordinate charts give $\ML(S)$ the structure of a PIL manifold, as asserted in the previous subsection.

\subsection{The mapping class group and pseudo-Anosov mappings}
The mapping class group is the collection of homotopy classes of topological symmetries of $S$, i.e., $\Mod(S) = \pi_0(\operatorname{Homeo}^+(S))$.
The natural action of $\Mod(S)$ on $\ML(S)$ by pushing forward measures is by PIL homeomorphisms.
Ahistorically, we call a mapping class $\varphi \in \Mod(S)$ \emph{pseudo-Anosov} if no proper power of $\varphi$ fixes any simple closed curve, i.e., for any essential simple closed curve $\gamma$,
\[ ~\varphi^n(\gamma) \sim \gamma \text{ implies $n=0$.}\]
The works of Nielsen and Thurston give the following equivalent description of pseudo-Anosov mapping classes.
There exists a scalar $\Lambda >1$ and measured geodesic laminations $\mu_{\pm} \in \ML(S)$ such that 
\[\varphi_*\mu_{+} = \Lambda \mu_{+} \text{ and } \varphi_* \mu_- = \Lambda\inverse \mu_-.\]
The $\varphi$-action on $\PML(S) = \ML(S) / \RR_{>0}$  has exactly two fixed points corresponding to the projective classes $[\mu_{+}]$ and $[\mu_-]$, which are attracting and repelling, respectively.

\subsection{The intersection and symplectic forms on $\ML(S)$}

There is a continuous \emph{geometric intersection form} 
\[i: \ML(S)\times \ML(S) \to \RR_{\ge 0}\]
that is homogeneous in each factor and that can be described by the formula
\[i(\mu, \nu) = \iint_S d\mu  d\nu,\]
where the measure $d\mu  d\nu$ is the local product measure supported on the set of transverse intersections of the supporting geodesic laminations.

The intersection is locally bi-linear; that is,  given measured laminations $\mu_1$ and $\mu_2$ with distinct supports, one can find maximal bi-recurrent train tracks $\tau_1$ and $\tau_2$ such that $\mu_i\prec \tau_i$ inducing positive weight systems on $\tau_i$, and the intersection pairing is bi-linear on $U(\tau_1) \times U(\tau_2)$.

There is also a symplectic form $\omega_{\Th}$ on $\ML(S)$, though one has to take care to interpret this correctly, as $\ML(S)$ does not have a natural smooth structure.
Rather, at each point $\mu\in \ML(S)$, there is a PL tangent cone $T_\mu\ML(S)$ consisting of one-sided tangent vectors that correspond to directions pointing into a train track coordinate chart where $\mu$ lies on the boundary.

For the (measurably) generic set of maximal measured laminations, the tangent cone  has a linear structure that can be identified with the weight space $W(\tau)$ of any train track $\tau$ such that $\mu\prec \tau$.  
More generally, there is always a subspace of the tangent cone with a well defined linear structure,  coming from the weight space of a train track that carries $\mu$ \emph{snugly}, i.e., that $\mu\prec \tau$ and that the interior of the complement of an arbitrarily small neighborhood of $\mu$ in $S$ has the same topology as the complement of $\tau$ in $S$.  
See \cite[Section 6]{Th_stretch} or \cite{Bon_GLTHB}. 

Let us consider a maximal measured lamination $\mu$ and a maximal bi-recurrent generic train track $\tau$ carrying $\mu$ so that $T_\mu\ML(S) \cong W(\tau)$.  
For vectors $\alpha, \beta \in W(\tau)$, we define the anti-symmetric bi-linear pairing
\[\omega_{\Th}(\alpha, \beta) = \frac{1}{2}\sum_{s\in s(\tau)} 
\begin{vmatrix}
\alpha(r_s) & \beta(r_s)\\
\alpha(\ell_s) & \beta(\ell_s)
\end{vmatrix},
\]
where at the trivalent switch $s$ there is an incoming branch and two outgoing branches $\ell_s$ and $r_s$ exiting $s$ to the left and to the right, respectively.
 An alternate description of $\omega_{\Th}$ as the homological intersection pairing on a certain double branched cover $\hat S \to S$ on which $\mu$ lifts to an orientable lamination $\hat \mu$ illustrates that in fact $\omega_{\Th}$ is non-degenerate, hence symplectic (where defined); see \cite[Section 3.2]{PennerHarer}.
See also \cite[Appendix]{BW:symplectic} for a description of the (degenerate) symplectic form on an arbitrary generic train track $\tau$.  
The action of the mapping class group preserves the symplectic form.

\subsection{Transverse H\"older distributions}\label{sub:THD}
To a geodesic lamination $\lambda \in\mathcal{GL}(S)$ and transversal $k$, we consider the space of H\"older distributions $\cH(k\cap \lambda)$ on $k\cap \lambda$; these are functionals on the space of H\"older continuous functions  on $k\cap \lambda$ that are bounded when restricted to functions with a fixed H\"older exponent.
A H\"older continuous isotopy transverse to $\lambda$ between two transversals $k$ and $k'$ induces a linear isomorphism between the H\"older functions and also the dual spaces, i.e.,  $\cH(k\cap \lambda) \cong \cH(k'\cap \lambda)$.  
A \emph{transverse H\"older distribution}  $\alpha$ to $\lambda$  is an assignment to each transversal $k$, a H\"older distribution on $k\cap \lambda$.  This assignment is required to be invariant under transverse H\"older isotopy and natural under inclusion.  

The vector space $\cH(\lambda)$ of transverse H\"older distributions to a given geodesic lamination has dimension
\[\dim_{\RR}\cH(\lambda) = -\chi(\lambda) +n_0(\lambda),\]
where $n_0(\lambda)$ is the number of orientable components of $\lambda$ and $\chi(\lambda)$ is its Euler characteristic.
The Euler characteristic can be computed via a train track snugly carrying $\lambda$  \cite[Section 5]{Bon_THDGL}.

There is a natural isomorphism between the weight space $W(\tau)$ of a snug train track $\tau$ for $\lambda$ and $\cH(\lambda)$.
Indeed, to $\alpha \in \cH(\lambda)$ this isomorphism associates the weight system $w_\alpha\in W(\tau)$ obtained  by pairing $\alpha$ with the constant function $1$ on a small transversal to each branch of $\tau$.  
This rule exhibits a transverse H\"older distribution as a ``kind of finitely additive signed transverse measure'' to $\lambda$ called a \emph{transverse cocycle} \cite[Sections 5 and 6]{Bon_THDGL}.

If $\lambda$ is maximal and measured, then by virtue of the identifications $T_\mu\ML(S)\cong W(\tau) \cong \cH(\tau)$, the space of H\"older distributions is the vector space of tangent directions to $\mu$ in $\ML(S)$, where the correspondence is given by 
\[\alpha \in \cH(\tau) \mapsto \dot \alpha =[\mu+ t\alpha]_{t\in [0,\epsilon)}\in T_\mu\ML(S).\]
Thus when $\mu$ is maximal and measured, the Thurston form $\omega_{\Th}$ gives $\cH(\mu)$ the structure of a $(6g-6)$-dimensional symplectic vector space.

Let $\mu\in \ML(S)$ be a measure whose support is contained in some $\lambda\in \mathcal {GL}(S)$. 
Let $\alpha \in \cH(\lambda)$, and consider the path $\mu+t\alpha \subset \cH(\lambda) $ with $t\in [ 0, \epsilon)$.
If $\mu+t\alpha$ gives all branches  non-negative mass for some train track $\tau$ carrying $\lambda$ snugly, then we can build a measured lamination corresponding to that weight system and attempt to extract a one-sided tangent vector $\dot \alpha = [\mu +t\alpha]_{t\in [0,\epsilon)}$.  
\begin{theorem}[Theorem 21 of \cite{Bon_GLTHB}]\label{thm:tangent}
With notation as above, $\alpha\in \cH(\lambda)$ corresponds to a tangent vector to a measure $\mu$ supported in $\lambda$ if and only if the following three conditions hold:
\begin{itemize}
\item every infinite isolated geodesic of $\lambda$ has non-negative $\alpha$-mass; and
\item every infinite isolated leaf of $\lambda$ which is asymptotic to a minimal sublamination of $\lambda$ that is not contained in the support of $\mu$ has $\alpha$-mass $0$; and
\item the restriction of $\alpha$ to each minimal sublamination of $\lambda$ that is not contained in the support of $\mu$ is a transverse measure.  
\end{itemize}
\end{theorem}

\subsection{Teichm\"uller space}
From now on, we will restrict ourselves to \emph{hyperbolic} metrics on $S$, i.e., those metrics with constant curvature everywhere equal to $-1$. 
The \emph{Teichm\"uller space} $\T(S)$ is the collection of equivalence classes of \emph{marked hyperbolic structures} on $S$.  
Note that we often suppress notation for the marking and write $Z\in \T(S)$.
The $\delta$-thick part of $\T(S)$ consists of those hyperbolic surfaces $Z$ where the injectivity radius $\inj(Z)\ge \delta$; equivalently, every closed geodesic loop has length least $2\delta$. 

To each $Z\in \T(S)$, there is a holonomy representation $\pi_1(S) \to  \Isom^+(\widetilde Z) \cong \PSL_2\RR$ that records the transitions between local isometric charts around the loops in $S$, well defined up to conjugation.
The image of $\T(S)$ in $\Hom(\pi_1(S), \PSL_2\RR)/\PSL_2\RR$ is a component of discrete and faithful representations and endows $\T(S)$ with the structure of an $\RR$-analytic manifold.

\subsection{The symplectic form on $\T(S)$}
Work of Goldman \cite{Goldman:symplectic}  provides us with a natural symplectic structure $\omega_{\text{G}}$ on the smooth locus in $\Hom(\pi_1(S), \PSL_2\RR)/\PSL_2\RR$.
It turns out that $\omega_{\text{G}}$ is a constant multiple of the Weil-Petersson K\"ahler form $\omega_{\WP}$ when restricted to $\T(S)$ \cite[Proposition 2.5]{Goldman:symplectic}. 

The symplectic form induces a bundle isomorphism $T^*\T(S) \to T\T(S)$, hence a linear isomorphism between  $1$-forms and vector fields.
Concretely, to a $1$-form $\eta$, there is a vector field $X$ defined by the rule that for all vector fields $Y$, we have  \[\omega_{\WP}(X, Y) = \eta(Y).\]

To a smooth function $f: \T(S)\to \RR$, we associate the \emph{Hamiltonian vector field} $X_f$ by the rule $\omega_{\WP}(X_f, \bullet) = df$.  
The symplectic form induces a \emph{Poisson bracket} on the space of smooth functions:
\[\{f, g\}(Z) = \omega_{\WP}(X_f(Z),X_g(Z)) = \left.\frac{d}{dt}\right|_{t=0} f(H_g(Z,t)) = \left. -\frac{d}{dt}\right|_{t=0} g(H_f(Z,t)),\]
where $H_g$ and $H_f$ are the (local) \emph{Hamiltonian flows} of the Hamiltonian vector fields $X_g$ and $X_f$, respectively. 

\subsection{Length}\label{sec:length}
To each point $Z\in \T(S)$, there is a homogeneous continuous length function 
\[\ell_\bullet (Z): \ML(S) \to \RR_{>0}\]
extending the rule that if $\mu$ corresponds to a weighted multi-curve then $\ell_\mu(Z)$ is the weighted sum of the lengths of the geodesic representatives on $Z$ of each component of the support of $\mu$.
Alternatively,  length is the integral of 
the product of the length element $d\ell$ along the geodesic leaves of $\mu$ in $Z$ and the transverse measure:
\[\ell_\mu(Z) = \iint_Z ~d\ell  d\mu.\]

For a simple closed curve $\gamma$ (with weight $1$), we have
\[2\cosh(\ell_\gamma(Z)/2) = |\tr(f_*(\gamma))|,\]
where $f_*: \pi_1(S) \to \PSL_2\RR$ is the holonomy representation corresponding to the marking $f: S\to Z$.  
Thus length functions for weighted multi-curves are manifestly $\RR$-analytic.  
Kerckhoff proved \cite{Kerckhoff:analytic} that $\ell_{\mu}: \T(S) \to \RR$ is analytic for every $\mu\in \ML(S)$.  

There is also a notion of length for transverse H\"older distributions.
Namely, for $Z\in \T(S)$, $\lambda\in \mathcal {GL}(S)$, and $\alpha\in \cH(\lambda)$, we define
\[\ell_\alpha(Z)= \iint_Z d\ell  d\alpha,\]
where $d\ell$ is the length element along leaves of $\lambda$ on $Z$, but since $\alpha$ is only finitely additive and signed, some care has to be used when interpreting the meaning of this integral; see \cite[\S6]{Bon_GLTHB} or \S\ref{sec:int}, below.

Let $\lambda\in \mathcal {GL}(S)$, and $\mu\in \ML(S)$ be a measure whose support is contained in $\lambda$.  For any $\alpha\in \cH(\lambda)$ representing a tangent direction to $\mu$ in $\ML(S)$, Bonahon proved \cite[Corollary 25]{Bon_GLTHB} that
\begin{equation}\label{eqn:dlength}
\ell_\alpha(Z) = \ddtp \ell_{\mu+t\alpha}(Z). 
\end{equation}

\subsection{Angle}\label{sec:angle}
The Poisson bracket between the length functions for  measured laminations $\mu$ and $\nu$ relates the change in length of one along the \emph{earthquake flow} in the other.
Formally, for $t\in \RR$, the earthquake flow $\Eq_{t \mu}: \T(S) \to \T(S)$  in $\mu$ is the (complete) $\RR$-analytic Hamiltonian flow associated to the length function $\ell_\mu$.  
Geometrically, to a hyperbolic surface $Z\in \T(S)$, the earthquake flow at time $t>0$ in a weighted curve $\gamma$ with weight $w>0$ is the surface $\Eq_{t\gamma}(Z)$ obtained by cutting open $Z$ along $\gamma$ and regluing with a twist of displacement $tw$ to the left.  
This rule extends continuously from weighted curves to measured laminations.

Kerckhoff \cite{Kerckhoff:NR} and Wolpert \cite{Wolpert:magic} proved that for weighted curves $\mu$ and $\nu$,
\[\{ \ell_\mu, \ell_\nu\}(Z) = \ddt \ell_{\mu}(\Eq_{t\nu}(Z)) =  \iint_Z \cos(\mu, \nu) \ d\mu  d\nu.\]
where  $\cos(\mu,\nu): Z\to [-1,1]$ is the function supported on the transverse intersection $\mu \pitchfork\nu$ measuring the cosine of the angle made from leaves of $\mu$ to leaves of $\nu$.
This formula extends continuously to arbitrary $\mu$ and $\nu\in \ML(S)$ \cite{Kerckhoff:NR}.

\subsection{Measured foliations}
	A (singular) measured foliation on a surface $S$ is a $C^1$ foliation $\mathcal F$ of $S$ away from a finite set (called singular points) and equipped with a transverse measure $\nu$.
	The transverse measure is required to be invariant under holonomy and every singularity is modeled on a standard $k$-pronged singularity; see  \cite{Thurston:bulletin, FLP} for details and further development. 
	Isotopic measured foliations are considered identical.
	
	The space  $\MF(S)$ of \emph{Whitehead equivalence classes} of singular measured foliations  
	is equipped with a topology coming from the geometric intersection number with homotopy classes of simple closed curves on $S$.

		In a given hyperbolic metric, we can pull each non-singular leaf of a measured foliation $(\mathcal F, \nu)$ tight in the universal cover to obtain a geodesic.  The closure  is a geodesic lamination invariant under the group of covering transformations that projects to a geodesic lamination $\lambda$ on $S$, and $\lambda$ carries a measure of full support $\mu$ obtained in a natural way from $\nu$.

This procedure defines a natural homeomorphism $\MF(S)\to \ML(S)$ \cite{Levitt:MFML}, so that we may pass between  equivalence classes of measured foliations and the corresponding measured lamination at will.

\subsection{The horocycle foliation}\label{sub:horofoliation}
The following construction is essentially due to Thurston \cite{Th_stretch,Thurston:notes}.
Let $P$ be a \emph{regular} ideal $n$-gon in $\HH^2$, i.e., $P$ has $n$-fold rotational symmetry.
There is a rotationally symmetric horocyclic $n$-gon $h_P$ whose edges are horocyclic arcs facing the ends of $P$. 
The unbounded components of $P\setminus h_P$ can be foliated by horocyclic segments facing the ideal points. 
This (partial) foliation $H(P)$ of $P$  is called the \emph{horocycle foliation} and it carries a transverse measure: homotope a transverse arc into the boundary of $P$ where it inherits Lebesgue measure.

Recall that  $\GL_\Delta(S)$ consists of minimal and filling geodesic laminations.
Given $\lambda\in \GL_\Delta(S)$, we define the \emph{regular locus}
\[\T^\lambda(S) = \{Z\in \T(S): \text{every component of $Z\setminus \lambda$ is regular}\}.\]
Note that if $\lambda$ is maximal, then $\T^\lambda(S) = \T(S)$, but otherwise $\T^\lambda(S)$ is an analytic submanifold of positive co-dimension.

For $\lambda\in \GL_\Delta(S)$ and $Z\in \T^\lambda(S)$, the horocycle foliation on each complementary component of $\lambda$ in $Z$ defines a partial foliation of $Z\setminus \lambda$.
The line field othogonal to $\lambda$ is locally Lipschitz (Lemma \ref{lem:lipschitz}), so the horocycle foliation extends uniquely across the measure $0$ set $\lambda$ to obtain a $C^1$ partial foliation $H_\lambda(Z)$ of $Z$.
The transverse measures on the complementary components piece together to a transverse measure on all of $H_\lambda(Z)$.
It is sometimes useful to view $H_\lambda(Z)$ as an element of $\MF(S)$ by collapsing each horocyclic $n$-gon to a point.
Indeed, this defines a map 
\[H_\lambda: \T^\lambda(S) \to \MF(S).\]
\begin{remark}
There is a similar construction of an \emph{orthogeodesic foliation} $\cO_\lambda(Z)$ which is defined for \emph{any} geodesic lamination $\lambda$ and \emph{any} $Z\in \T(S)$.
When $\lambda \in \GL_\Delta(S)$ and $Z\in \T^\lambda(S)$, then $\cO_\lambda(Z)$ and $H_\lambda(Z)$ agree in $\MF(S)$.
See \cite{CF} for details.
\end{remark}

\subsection{Shearing coordinates}\label{sub:shear}
Configurations of ideal triangles  $T_1, T_2$ glued along an edge in $\HH^2$ are parameterized by their \emph{shear coordinate} $\sigma(T_1, T_2)$: a signed distance between the tips of the two interior horocyclic triangles along their shared geodesic edge.  The sign is positive if the tip of the horocyclic triangle in $T_2$ is to the left of the tip of the horocyclic triangle in $T_1$ when viewed from $T_1$; this assignment is symmetric, i.e., $\sigma(T_1, T_2) = \sigma(T_2, T_1)$.

Let $\lambda$ be a maximal geodesic lamination, and $Z\in \T(S)$.  
The horocycle foliation $H_\lambda(Z)$ gives us a way to identify the closest edges $g_1$ and $g_2$ of two ideal triangles $T_1$ and $T_2$ in the universal cover in an isometric  way, even when they are not glued along an edge (see  \cite[\S2]{Bon_SPB} or \cite[\S13.2]{CF}).
The shear parameter $\sigma_\lambda(Z)(T_1, T_2)$ is defined similarly for triangles glued along an edge, and is again symmetric.

To a train track $\tau$ carrying $\lambda$, the rule assigning $\sigma_\lambda(Z)(T_1, T_2)$ to triangles $T_1$ and $T_2$ that are adjacent across a branch $\tilde b \subset \widetilde \tau$ in $\widetilde Z$ defines a weight system on $\widetilde \tau$ that is invariant under the covering group.  
Thus $\sigma_\lambda(Z)$ defines an element of $W(\tau) \cong \cH(\lambda)$, and we have  a map \[\sigma_\lambda: \T(S) \to \cH(\lambda)\] called \emph{shear coordinates} adapted to $\lambda$.

We give a brief summary of some of the relevant (remarkable) properties of the shear coordinates, due to Bonahon \cite{Bon_SPB} (see also \cite{Bon_GLTHB, Bon_THDGL} and \cite{BonSoz}).

\begin{theorem}\label{thm:coordinates}
Given a maximal geodesic lamination $\lambda$, the following are true:
\begin{itemize}
\item For a measure $\mu \in \ML(S)$ with support contained in $\lambda$ and $Z\in \T(S)$, we have \[\ell_\mu(Z) = \omega_{\Th} ( \mu, \sigma_\lambda(Z)).\]
\item More generally, for any $\alpha\in \cH(\lambda)$, we have $\ell_\alpha(Z) = \omega_{\Th} (\alpha,\sigma_\lambda(Z) )$.
\item  The  map $\sigma_\lambda$ is a real analytic diffeomorphism onto its image $\cH^+(\lambda)\subset \cH(\lambda)$.
\item The set $\cH^+(\lambda)$ is a convex cone of \emph{positive} distributions, defined by the intersection of at most $3g-3$ linear inequalities $\omega_{\Th}( \mu,\cdot ) >0$, where $\mu$ is an ergodic measure whose support is contained in $\lambda$.
\item The earthquake flow in a measure $\mu \in \ML(S)$ supported in $\lambda$ is linear in shear coordinates:
\[\sigma_\lambda(\Eq_{t\mu}(Z)) = \sigma_\lambda(Z) + t\mu. \]
\item The pullback of the Thurston symplectic form $\omega_{\Th}$ on $\cH^+(\lambda)$ via $\sigma_\lambda$ is equal to (a constant multiple) of the Weil-Petersson symplectic form $\omega_{\WP}$ on $\T(S)$ \cite{BonSoz}. 
\end{itemize}
\end{theorem}

If $\mu\in \ML(S)$ is a measure whose support is contained in $\lambda$ and $\alpha\in \cH(\lambda)$ represents a tangent direction $\dot \alpha \in T_\mu\ML(S)$,  then combining the second bullet point in Theorem \ref{thm:coordinates} with Equation \eqref{eqn:dlength} gives 
\[\ell_\alpha (Z) =\omega_{\Th}(\alpha,\sigma_\lambda(Z) ) =  \ddtp \ell_{\mu+t\alpha}(Z).\]

For any direction $\alpha \in \cH(\lambda)$, there is an analytic \emph{shearing vector field} $X_\alpha$ on $\T(S)$ given by 
\[\left(\sigma_\lambda\right)_*X_\alpha = \alpha,\] where we identify the vector space $\cH(\lambda)$ with its tangent space at every point.
The following lemma follows directly from Theorem \ref{thm:coordinates} and says that length functions are Hamiltonian for shearing vector fields.
\begin{lemma}\label{lem:length_dual}
Let $\lambda$ be a maximal geodesic lamination, and let $\alpha \in \cH(\lambda)$.
The shearing vector field $X_\alpha$ is dual to $d\ell_\alpha$ with respect to the Weil-Petersson symplectic form.  That is, $d\ell_\alpha  = \omega_{\WP}(X_\alpha, \cdot)$.  
\end{lemma}

Thurston's \emph{stretch flow} also has a nice formulation in shear coordinates.
Stretch lines are certain kinds of directed geodesics for Thurston's asymmetric Lipschitz metric on $\T(S)$ obtained by gluing together \emph{stretch maps} on the ideal triangle along a maximal geodesic lamination $\lambda$; see \cite[\S4]{Th_stretch} for details.

If $\operatorname{stretch}(Z,\lambda,t)\in \T(S)$ is the hyperbolic metric obtained by stretching $Z$ along $\lambda$ for time $t$, then we have
\[\sigma_\lambda(\operatorname{stretch}(Z,\lambda,t)) = e^t\sigma_\lambda(Z).\]

\begin{remark}
There is also a notion of a \emph{generalized stretch flow} defined for points in the regular locus $\T^\lambda(S)$ for a (non-maximal) filling lamination $\lambda$ obtained by gluing together stretch maps on the regular ideal polygons in the complement of $\lambda$; the construction of generalized stretch maps is a straightforward generalization of the  construction of stretch maps.
See also \cite[\S15]{CF} for the construction of the \emph{dilation flow} which is defined for any  $(Z,\mu,t) \in \T(S) \times \ML(S)\times \RR$ and specializes to generalized stretch flow when $\mu$ is filling and regular on $Z$.
\end{remark}

\section{Hamiltonian Flows in shear coordinates}\label{sec:flow}

The goal of this section is to formulate and prove a more general version of Theorem \ref{thm:pA_Hamiltonian} from the introduction.
First we describe the linear action of (a bounded positive power of) $\varphi$ on the image of the shearing coordinates adapted to a (maximal completion) of the support of the attractive fixed point for $\varphi$ on $\PML(S)$.

Once we have the linear action of $\varphi$ in hand, the proof reduces to a computation in linear algebra in shear coordinates.
We then account for general Jordan decompositions of the linear action of $\varphi$.

\subsection{The linear action of a pseudo-Anosov mapping class}
Let $\varphi \in \Mod(S)$ be a pseudo-Anosov mapping class with projectively invariant measured laminations $[\mu_\pm]\in \PML(S)$.
Choose a lift $\mu_+$ in the class of $[\mu_+]$ and   $\mu_-$ lifting $[\mu_-]$ such that $i(\mu_+,\mu_-) = 1$. 

To understand the action of $\varphi$ near its attracting fixed point $[\mu_+]\in \PML$, we find a generic \emph{invariant train track} $\tau_0$ for $\varphi$.  That is, $\phi(\tau_0) \prec \tau_0$ where $\phi$ is a homeomorphism representing $\varphi$.
We also require that $\tau_0$ carry $\mu_+$ snugly so that $\tau_0$ is not maximal if $\mu_+$ is not maximal. 

A generic bi-recurrent invariant train track $\tau_0$ exists and  there is a choice of carrying map $h$ for $\phi(\tau_0) \prec \tau_0$ such that if $M_h$ is the incidence matrix obtained via \eqref{eqn:incidence}, then 
\[M_\phi: \RR^{b(\tau_0)} \xrightarrow{\phi}\RR^{b(\phi(\tau_0))} \xrightarrow{M_h} \RR^{b(\tau_0)}\]
is primitive irreducible \cite[Theorem 4.1]{PP:pA}.
The Perron-Frobenius eigenvalue of $M_\phi$ is $\Lambda$ and its corresponding eigenvector is the weight system on $\tau_0$ corresponding to $\mu_+$.
The restriction of $M_\phi$ to weight spaces 
\[A_\varphi: W(\tau_0) \to W(\tau_0)\]
only depends on $\varphi$.

If $\tau_0$ was not maximal, we now choose a generic maximal bi-recurrent completion $\tau$ which is invariant under a bounded power of $\varphi$.
That is, $\tau$ is obtained from $\tau_0$ by attaching branches to $\tau_0$ in the complementary components of $\tau_0$ that were not triangles, and for some $1\le k\le k_0(g)$, we have $\phi^k(\tau) \prec \tau$.
To make this choice, observe that $\phi$ permutes the complementary components of $\tau$,\footnote{We remark that if $\varphi$ was obtained from Penner's construction \cite{Penner:pA}, then $\varphi$ preserves not only the components of the complement of $\mu_+$ but fixes all spikes.} so some power bounded by the complexity of $S$ fixes the spikes of the complementary polygons.
Any choice of ideal triangulation of the ideal polygons complementary to the support of $\mu_+$ defines a maximal  geodesic lamination $\lambda_+$ that contains the support of $\mu_+$ as a sub-lamination; make this choice so that $\lambda_+$ is chain recurrent.  This induces a choice of additional branches of $\tau_0$ which can be combed apart so that again all switches are trivalent.

The resulting train track $\tau$ is generic bi-recurrent maximal and invariant, so we have an induced map of weight spaces 
\[B_{\varphi^k}: W(\tau) \to W(\tau)\]
that restricts to $A_\varphi ^k$ on the subspace $W(\tau_0)$.

Then $B_{\varphi^k}$ preserves the integer lattice and Thurston symplectic form $\omega_{\Th}$ on $W(\tau)$.  
Up to replacing $k$ with $2k$ (without changing notation),  the set of eigenvalues of $B_{\varphi^k}$ contains no negative real numbers.

By Theorem \ref{thm:tangent}, there is a convex cone with finitely many sides in $W(\tau)\cong \cH(\lambda_+)$ that identifies with a linear fragment of $T_{\mu_+}\ML(S)$ containing $W(\tau_0) \cong \cH(\mu_+)$ as a (vector) subspace.  
Since $B_{\varphi^k}$ describes the action of $\varphi^k$ on $\ML(S)$ in the PIL chart $W(\tau)$, $B_{\varphi^k}$ also describes the linearization on the directions in $\T_{\mu_+}\ML(S)$ coming from $\cH(\lambda_+)$.  

\subsection{Duality}\label{sec:duality}
We relate the spectrum of the operator $B_{\varphi^k}$ to the spectrum of a dual linear operator modeling the action of $\varphi^{-k}$ on certain directions near $\mu_-$.
This is essentially just a summary/repackaging of \cite[\S3.4 and  Epilogue]{PennerHarer}.

To a maximal train track $\tau$ there is a \emph{dual bi-gon track} $\tau^*$, such that $\phi^k(\tau) \prec \tau$ implies  $\phi^{-k} (\tau)^* \prec \tau^*$.  
This induces an integer linear isomorphism  
\[B'_{\varphi^{-k}}: W'(\tau^*) \to W'(\tau^*),\]
where $W'(\tau^*)$ is a linear quotient of the subspace $W(\tau^*)\le \RR^{b(\tau^*)}$ satisfying the switch conditions obtained as follows.
For every branch $b$ of $\tau$ there is a dual branch $b^*$ of $\tau^*$ meeting $b$ once, and for vectors $u \in \RR^{b(\tau)}$ and $v\in \RR^{b(\tau^*)}$ a non-degenerate bilinear pairing
\[\langle u,v \rangle = \sum_{b\in b(\tau)} u(b) \cdot v(b^*).\]
The paring induces a linear mapping 
\[\iota: W(\tau^*) \to W(\tau)^*,\]
which is surjective when $\tau$ is maximal.  By definition, we have $W'(\tau^*) = W(\tau^*)/\ker(\iota)$ so that $W'(\tau^*) \cong W(\tau)^*$ holds. 
Moreover, there is a \emph{positive cone} in $W'(\tau^*)$ that defines a chart for $\ML(S)$ in a natural way that contains $\mu_-$ (on a facet if $\tau$ was not maximal).
As such, there is a (unique) chain recurrent geodesic lamination $\lambda_-$ invariant under $\varphi^{-k}$ that contains the support of $\mu_-$ as a minimal sub-lamination and a natural isomorphism
$ W'(\tau^*) \cong \cH(\lambda_-)$ conjugating the linear $\varphi^{-k}$ actions.

The Thurston form $\omega_{\Th}$ on $W(\tau)$ then induces an isomorphism $W(\tau)^* \to  W(\tau)$.
Composing isomorphisms, we have 
\[\ast : W'(\tau^*) \to W(\tau)^*\to W(\tau),\]
which satisfies 
\[\omega_{\Th}( w, \ast [v]) = \langle w,v\rangle. \]
If $[v]$ and $w$ both lie in their respective positive cones, then \[\omega_{\Th}(w,\ast [v]) = i( \mu_{w}, \mu_{[v]}),\] for the corresponding measured laminations.

Finally, the relationship between $B_{\varphi^k}$ and $B_{\varphi^{-k}}'$ is 
\[\ast B'_{\varphi^{-k}}[v] = B_{\varphi^k}\inverse \ast [v].\]
In particular, the spectra of $B'_{\varphi^{-k}}$ and $B_{\varphi^{k}}$ coincide.
Moreover, direct computation shows that $\ast$ is a symplectomorphism with respect to the natural symplectic structure on $W'(\tau^*)$, so that if $\alpha_1, ..., \alpha_{3g-3}, \beta_1, ...,\beta_{3g-3}$ is a symplectic basis for  $W(\tau)$, then $\ast\inverse \alpha_1, ..., \ast\inverse\alpha_{3g-3},$  $\ast\inverse\beta_1, ..., \ast\inverse\beta_{3g-3}$ is a symplectic basis  for  $W'(\tau^*)$.

Summarizing, the action of $\varphi^k$ on $\ML(S)$ near $\mu_+$ is naturally symplectomorphic to the action of $\varphi^{-k}$ on $\ML(S)$ near $\mu_-$.

\subsection{Hamiltonian flows for pseudo-Anosov mappings with simple spectrum}\label{sec:flow_simple}

The spectrum of a symplectic linear map is symmetric about $1$; the eigenvalues come in reciprocal and complex conjugate pairs with the same multiplicity.

Recall that we chose $k$ so that $B_{\varphi^k}$ has no negative real eigenvalues.
We list the eigenvalues of $B_{\varphi^k}$ according to their algebraic multiplicity;\footnote{The algebraic multiplicity of a complex eigenvalue $\Lambda_i$ is the number of times it appears on the main diagonal in the Jordan decomposition over $\CC$.} let
\begin{itemize}
\item $\Lambda_\RR^+=\{\Lambda^k = \Lambda_1, \Lambda_2, ..., \Lambda_p\}$ denote the real eigenvalues of $B_{\varphi^k}$ larger than $1$;
\item $\Lambda_{\mathbb T}^+= \{\Lambda_{p+1}, ..., \Lambda_{p+q}\}$ denote the eigenvalues on the unit circle with argument in $[0,\pi)$;
\item $\Lambda_\HH^+ = \{\Lambda_{p+q+1}, ..., \Lambda_{p+q+r}\}$ denote the (complex) eigenvalues lying in the upper half plane with magnitude larger than $1$.
\end{itemize}

\begin{lemma}\label{lem:simple_jordan}
If $B_{\varphi^k}$ is diagonalizable over $\CC$, then  there is a symplectic basis of (generalized) eigenvectors \[\alpha_1, ..., \alpha_{p+q}, \alpha_{p+q+1}^\pm, ..., \alpha_{p+q+r}^\pm, \beta_1, ..., \beta_{p+q}, \beta_{p+q+1}^\pm, ..., \beta_{p+q+r}^\pm\in\cH(\lambda_+)\] that satisfies

\begin{itemize}
\item For $i = 1, ..., p$, we have 
\[
\begin{matrix}
B \alpha_i = \Lambda_i \alpha_i & \text{ and } & B\beta_i = \Lambda_i\inverse \beta_i.
\end{matrix}
\]
\item For $i = p+1, ..., p+q$, we have $\Lambda_i = e^{i\theta_i}$ for $\theta_i \in [0,\pi)$ and 
\[
\begin{matrix}
B\alpha_i = \cos(\theta_i) \alpha_i +\sin(\theta_i)\beta_i &\text{ and }& B \beta_i = \cos(\theta_i) \beta_i - \sin(\theta_i)\alpha_i.
\end{matrix}
\] 

\item For $i = p+q+1, ..., p+q+r$, we have $\Lambda_i=|\Lambda_i|e^{i\theta_i}$ for $\theta_i \in (0,\pi)$ and 
\[\begin{matrix}
B \alpha_i^+ =  |\Lambda_i|\left( \cos(\theta_i) \alpha_i^+ + \sin(\theta_i) \alpha_i^- \right)& & B\alpha_{i}^- = |\Lambda_i|\left( \cos(\theta_i) \alpha_{i}^- -\sin(\theta_i) \alpha_{i}^+\right) \\
&\text{ and }&\\
B \beta_i^+ =  \left|\Lambda_i\right|\inverse\left( \cos(\theta_i) \beta_i^+ -  \sin(\theta_i)\beta_{i}^- \right)& & B \beta_{i}^- =   \left|\Lambda_i\right|\inverse\left(\cos(\theta) \beta_{i}^- +\sin(\theta_i) \beta_{i}^+\right).
\end{matrix}\]

\end{itemize}
The subspaces $\langle \alpha_1, ..., \alpha_{p+q}, \alpha_{p+q+1}^\pm, ..., \alpha_{p+q+r}^\pm\rangle$ and $\langle \beta_1, ..., \beta_{p+q}, \beta_{p+q+1}^\pm, ..., \beta_{p+q+r}^\pm\rangle$ form a Lagrangian splitting of $\cH(\lambda_+)$ with $\omega_{\Th}(\alpha_i, \beta_j) = \delta_{ij}$, $\omega_{\Th}(\alpha_i^\pm, \beta_j^\pm) = \delta_{ij}$, and $\omega_{\Th}(\alpha_i^\mp, \beta_j^\pm) = 0$.
\end{lemma}
\begin{proof}
This is the real Jordan normal form for (symplectic) linear maps diagonalizable over $\CC$.  See also \cite{Symplectic:canonical}.
\end{proof}

Recall from \S\ref{sec:length} that to each $\alpha\in \cH(\lambda_+)$ there is an analytic length function $\ell_\alpha: \T(S) \to \RR$, which is Hamiltonian for the corresponding shear vector field $X_\alpha$ and the Weil-Petersson symplectic form (Lemma \ref{lem:length_dual}). 

With this notation set, we can restate and prove a more general version of Theorem \ref{thm:pA_Hamiltonian}.
The proof is an elementary computation in symplectic linear algebra in shear coordinates.
\begin{theorem}\label{thm:simple_spectrum}
Let $\varphi \in \Mod(S)$ be a pseudo-Anosov mapping class with stretch factor $\Lambda>1$ and (projectively) invariant $\mu_{\pm}\in \ML(S)$ satisfying 
\[\varphi_*\mu_+ = \Lambda \mu_+ \text{ and } \varphi_*\mu_- = \Lambda\inverse \mu_-.\]

Let $\lambda_+$ be a chain recurrent maximal geodesic lamination that contains the support of $\mu_+$ as a sub-lamination, and let $1\le k\le k_0(g)$ be such that $\varphi^k$ preserves $\lambda_+$.
Assume that the linear action of $\varphi^k$ on $\cH(\lambda_+)$ is diagonalizable over $\CC$ with no negative real eigenvalues.

For the symplectic bases of eigenvectors afforded by Lemma \ref{lem:simple_jordan}, the time $1$ flow of the real-analytic Hamiltonian function 
\begin{align*}
F^{\lambda^+}_{\varphi^{-k}} = \sum_{i = 1}^{p} \log(\Lambda_i) ~\ell_{\alpha_i}\ell_{\beta_i}+ &\sum_{i = p+1}^{p+q}\frac{\theta_i}{2}(\ell_{\alpha_i}^2 + \ell_{\beta_i}^2) \\
&+ \sum_{i = p+q+1}^{p+q+r} \log(|\Lambda_i|)\left(\ell_{\alpha_i^+}\ell_{\beta_i^+}+\ell_{\alpha_i^-}\ell_{\beta_i^-}\right) + \theta_i\left(\ell_{\alpha_i^-}\ell_{\beta_i^+}-\ell_{\alpha_i^+}\ell_{\beta_i^-}\right) 
\end{align*}
coincides with the action of $\varphi^{-k}$ on $\T(S)$.

Moreover, the unique $\varphi^{-k}$ invariant Thurston stretch line forward-directed by $\lambda_+$ is a flow line parameterized proportionally to directed arclength.
\end{theorem}
\begin{remark}
There are pseudo-Anosov mapping classes whose projectively invariant measured laminations $\mu_+$ and $\mu_-$ are both maximal, and the linear action $A_\varphi$ on $W(\tau_0)\cong \RR^{6g-g}$ has simple spectrum with no eigenvalues on the unit circle. 
Indeed, the stretch factor $\Lambda$ of mapping classes arising from Penner's construction \cite{Penner:pA} have no Galois conjugates on the unit circle \cite{HS:coronal}, and it is possible to use Penner's construction to obtain $\Lambda$ whose degree over $\QQ$ is $6g-6$ \cite{Strenner:degree}.
Thus the characteristic polynomial of $A_\varphi$ agrees with the minimal polynomial of $\Lambda$.  Since irreducible polynomials have distinct roots, this provides examples as claimed at the beginning of the remark.

For such examples, we have $\tau_0 = \tau$, $\lambda_\pm = \mu_\pm$, $T_{\mu_\pm}\ML(S)\cong \cH(\mu_\pm)$, and $B_{\varphi^k}=A_{\varphi^k}$ for $k = 1,2$ (corresponding to whether or not $A_\varphi$ has any negative eigenvalues).
\end{remark}
\begin{remark}
Theorem \ref{thm:simple_spectrum} applied to $\varphi^{-k}$ which preserves the (dual) maximal lamination $\lambda_-$ whose support contains $\mu_-$ provides us with a very similar function $F^{\lambda_-}_{\varphi^k}$ whose Hamiltonian flow at time one coincides with the action of $\varphi^k$. The unique invariant stretch line forward directed by $\lambda_-$ is a flow line parameterized proportionally to arclength.  

From the discussion in \S\ref{sec:duality}, there is a dual symplectic basis of (generalized) eigenvectors $\ast\inverse \alpha_i$, $\ast\inverse\alpha_i^\pm$, $\ast\inverse \beta_i$ and $\ast\inverse\beta_i^\pm \in \cH(\lambda_-)$ compatible with the action of $B_{\varphi^{-k}}'$ on $\cH(\lambda_-)$. 
The function $F^{\lambda_-}_{\varphi^k}$ is obtained from $F^{\lambda_+}_{\varphi^{-k}}$ by negation and replacing each length function $\ell_{\alpha_i^\bullet}$ with the corresponding length function $\ell_{\ast\inverse\alpha_i^\bullet}$ for all decorations $\bullet$, and similarly with the $\beta$'s.
\end{remark}

To prove Theorem \ref{thm:simple_spectrum}, we leverage the following elementary lemma from symplectic linear algebra.
\begin{lemma}\label{lem:form}
Let $(V,\omega)$ be a symplectic vector space and $A \in \Sp(V)$ a symplectic linear map.  Let $\mathsf{sp}(V)$ denote the lie algebra of $\Sp(V)$ and suppose we have $X\in \mathsf{sp}(V)$ with $\exp(X) = A$.

Then $A$ is the Hamiltonian flow at time one of  the quadratic form \begin{align*}
F: V&\to \RR \\
p&\mapsto \frac12\omega(Xp, p).
\end{align*}
\end{lemma}

\begin{proof}
The elements $X\in \mathsf{sp}(V)$ satisfy $\omega(Xv,w) = -\omega(v,Xw)$, so the map $(v,w) \mapsto \omega(Xv,w)$ is symmetric and bilinear.
The time $t$ flow of the vector field $p \mapsto Xp \in T_pV\cong V$ is given by \[p\mapsto\exp(tX)p.\]
This means we just need to show that $dF(p)v = \omega(Xp, v)$ for all $v$.
We compute
\begin{align*}
dF(p)v &= \frac12\left(\omega(Xp,v) + \omega(Xv,p)\right)\\
& =\frac12\left(\omega(Xp,v)-\omega(v,Xp)\right)\\
&=\omega(Xp,v),
\end{align*}
which is what we wanted to show.
\end{proof}

\begin{proof}[Proof of Theorem \ref{thm:simple_spectrum}]
The first observation we make is that 
\[\varphi^k\left(\sigma_{\lambda_+}(Z)\right) =B_{\varphi^k}\sigma_{\lambda_+}(Z).\]

Let us examine the coordinate functions $x_i, y_i, x_i^\pm,$ and $y_i^\pm$ for $\sigma_{\lambda_+}(Z)$ with respect to our chosen symplectic basis adapted to $B_{\varphi^k}$:
\[\sigma_{\lambda_+}(Z) = \sum x_i \alpha_i + y_i \beta_i + \sum x_i^\pm \alpha_i^\pm + y_i^\pm\beta_i^\pm.\]
Using Theorem \ref{thm:coordinates}, we see
\[
\begin{matrix}
\ell_{\alpha_i^\bullet}(Z) = \omega_{\Th}(\alpha_i^\bullet,\sigma_{\lambda_+}(Z)) = y_i^\bullet & \text{ and } &  \ell_{\beta_i^\bullet}(Z) = \omega_{\Th}(\beta_i^\bullet,\sigma_{\lambda_+}(Z)) =- x_i^\bullet ,
\end{matrix}\]
where $\bullet$ represents one of the decorations `$+$', `$-$', or `$~$'.
This means that
\begin{equation}\label{eqn:shears}
\sigma_{\lambda_+}(Z) = \sum -\ell_{\beta_i} \alpha_i + \ell_{\alpha_i} \beta_i + \sum - \ell_{\beta_i^\pm} \alpha_i^\pm +\ell_{\alpha_i^\pm}\beta_i^\pm. 
\end{equation}

Comparing with Lemma \ref{lem:simple_jordan}, we specify an infinitesimal symplectic transformation  $X\in \mathsf{sp}(\cH(\lambda_+))$ with $\exp(X) = B_{\varphi^k}$ by the rules
\begin{itemize}
\item For $i = 1, ..., p$, 
\[
\begin{matrix}
X\alpha_i = \log(\Lambda_i) \alpha_i & \text{ and } & X\beta_i = -\log(\Lambda_i)\beta_i
\end{matrix}
\]
\item For $i = p+1, ..., p+q$, 
\[
\begin{matrix}
X\alpha_i = \theta_i \beta_i & \text{ and } & X\beta_i =- \theta_i\alpha_i
\end{matrix}
\]
\item For $i = p+q+1, ..., p+q+r$, 
\[
\begin{matrix}
X\alpha_i^+ = \log(|\Lambda_i|)\alpha_i^+ + \theta_i \alpha_i^- &  & X\alpha_i^- = \log(|\Lambda_i|)\alpha_i^- - \theta_i \alpha_i^+\\
& \text{ and } & \\
X\beta_i^+ = -\log(|\Lambda_i|)\beta_i^+ -\theta_i \beta_i^- & & X\beta_i^- = -\log(|\Lambda_i|)\beta_i^- +\theta_i \beta_i^+.
\end{matrix}
\]
\end{itemize}

By Lemma \ref{lem:form}, the Hamiltonian flow at time $1$ of the quadratic function in the shear coordinate functions
\begin{equation}\label{eqn:form}
Z \mapsto \frac12 \omega_{\Th}( X\sigma_{\lambda_+}(Z), \sigma_{\lambda_+}(Z))
\end{equation}
coincides with the action of $B_{\varphi^k}$ on $\cH(\lambda_+)$, hence the action of $\varphi^k$ on $\T(S)$ because $\sigma_{\lambda_+}$ is a real analytic symplectomorphism onto its image.

For geometric reasons, we wish to run this flow backwards to obtain the $\varphi^{-k}$ action on $\T(S)$, so we negate the function \eqref{eqn:form} so that its flow at time one is the linear map $B_{\varphi^k}\inverse$ on $\cH(\lambda_+)$.
Expanding this formula using \eqref{eqn:shears} yields the function $F^{\lambda_+}_{\varphi^{-k}}: \T(S) \to \RR$ from the statement of the theorem.

Now we explain the geometric rational for reversing the flow generated by \eqref{eqn:form}.
For any $Z\in \T(S)$, the Thurston stretch path directed by $\lambda_+$ through $Z$ is \[t\mapsto \sigma_{\lambda^+}\inverse(e^t\sigma_{\lambda_+}(Z)),\]
so that the stretch paths are the rays $e^t\sigma$ emanating from zero that satisfy the positivity condition $\omega_{\Th}(\mu, e^t \sigma)>0$ for all measures $\mu$ whose support is contained in $\lambda_+$.  
By unique ergodicity of $\mu_+$ and the fact that the isolated leaves of $\lambda_+$ cannot carry any measure, the positivity condition is just $\omega_{\Th}(\mu_+, e^t \sigma)>0$.

The path $t\mapsto Z_t$ defined by 
\[Z_t = \sigma_{\lambda^+}\inverse(e^t \beta_1)\]
is the Thurston stretch path forward directed by $\lambda_+$ parameterized according to  directed arclength, where the length of $\mu_+$ at time $t$ is 
\[\ell_{\mu_+}(Z_t) = \omega_{\Th}(\mu_+, e^t\beta_1) = \omega_{\Th}(\alpha_1, e^t\beta_1) =e^t.\]

No other eigenvector of $B_{\varphi^k}$ besides $\beta_1$ is positive.  Indeed, all other (generalized) eigenvectors $v$ satisfy $\omega_{\Th}(\mu_+,v) = \omega_{\Th}(\alpha_1,v)=0$.

We see that $\exp(tX)\beta_1 = \exp(-tk\log(\Lambda))\beta_1$, while $\exp(-tX)\beta_1 = \exp(tk\log(\Lambda))\beta_1$.
Thus the flow generated by $F_{\lambda_+}$  (which corresponds to $-X$) has $t\mapsto Z_{tk\log(\Lambda)}$ as a flow line, which is parametrized proportionally to the forward directed arclength (which is typically different from arclength with time reversed).  
\end{proof}

\subsection{More general Jordan blocks}\label{sec:flow_general}
In this section, we follow the proof of Theorem \ref{thm:simple_spectrum} to produce a (real valued) quadratic polynomial in length functions for each Jordan block in the normal form of $B_{\varphi^k}$ over $\CC$.
These quadratic polynomials poisson-commute with one another, and so summing over the Jordan blocks produces the desired function whose time one Hamiltonian flow induces $\varphi^{-k}$ on $\T(S)$.

For this we use the real canonical forms of Laub and Meyer \cite{Symplectic:canonical} to obtain a symplectic basis for the symplectic subspaces associated to certain $B_{\varphi^k}$-invariant subspaces in which $B_{\varphi^k}$ (or its logarithm) is particularly nice looking.
Note however that our notation does not agree with theirs. 
In what follows, we only treat Jordan blocks of size at least two, as Theorem \ref{thm:simple_spectrum} handles Jordan blocks of size one.

\para{Case: $\Lambda_i\in \Lambda_\RR^+$}
Suppose there is real eigenvalue $\Lambda_i\in \Lambda_{\RR}$ that has a Jordan block of size $k$ in the normal form over $\CC$.
That is, the $k$-by-$k$ matrix
\[ \begin{pmatrix}
\Lambda_i & 1 & & &\\
& \Lambda_i & 1 & &  \\
& &\ddots & & 1\\
& & & &\Lambda_i 
\end{pmatrix}
\]
appears on the diagonal for the Jordan normal form of $B_{\varphi^k}$.

Let $V(\Lambda_i)$ be the invariant subspace associated to this block (this is an abuse of notation as there might be several Jordan blocks with $\Lambda_i$ on the main diagonal).
Since $B_{\varphi^k}$ preserves a symplectic form, there is another Jordan block of size $k$ corresponding to $\Lambda_i\inverse$ with preserved subspace $V(\Lambda_i\inverse)$, where $V(\Lambda_i)$ and $V(\Lambda_i\inverse)$ are Lagrangians in the $2k$-dimensional symplectic subspace $W(\Lambda_i)$ they span \cite[\S\S2-3]{Symplectic:canonical}.

The normal form \cite[(1)]{Symplectic:canonical} gives us a symplectic basis $\alpha_1, ..., \alpha_k, \beta_1, ..., \beta_k$ of $W(\Lambda_i)\le \cH(\lambda_+)$ and a standard linear vector field $X\in \mathsf{sp}(W(\Lambda_i))$ with $\exp(X) = B_{\varphi^k}|_{W(\Lambda_i)}$. 
Writing $\sigma$ for the projection of $\sigma_{\lambda_+}(Z)$ to $W(\Lambda_i)$ and $\omega$ the restriction of $\omega_{\Th}$, we compute
\[-\frac 12 \omega(X\sigma, \sigma)=  \log(\Lambda_i)  \sum_{j = 1}^k\ell_{\alpha_j}\ell_{\beta_j} + \sum_{j = 1}^{k-1} \ell_{\alpha_{j+1}} \ell_{\beta_{j}}. \]
This is the desired quadratic polynomial for this Jordan block.  

\para{Case: $\Lambda_i\in \Lambda_\HH^+$}
This case is similar; we proceed with analogous notation as in the previous case.  
We have a complex eigenvalue of $B_{\varphi^k}$ with a Jordan block of size $k$ in the normal form over $\CC$.
We consider then the \emph{real} Jordan normal form, in which we can find a $2k$-by-$2k$ real Jordan block of the form
\[ \begin{pmatrix}
L_i & I & & &\\
& L_i & I & &  \\
& &\ddots & & I\\
& & & &L_i
\end{pmatrix}
\]
where 
\[L_i = |\Lambda_i| \begin{pmatrix}
\cos(\theta_i) & -\sin(\theta_i) \\
\sin(\theta_i) & \cos(\theta_i)
\end{pmatrix} \text{ and } I = \begin{pmatrix}
1&0 \\
0& 1
\end{pmatrix}.\]
Let $V(\Lambda_i)$ be the corresponding invariant subspace. 

Since $B_{\varphi^k}$ is symplectic, there is a real Jordan block of size $2k$ corresponding to $\Lambda_i\inverse$ and its complex conjugate with invariant subspace $V(\Lambda_i\inverse)$. 
The subspace $W(\Lambda_i)\le \cH(\lambda_+)$ spanned by these two has dimension $4k$ and $V(\Lambda_i)\oplus V(\Lambda_i\inverse)$ is a Lagrangian splitting.

The normal form \cite[(2)]{Symplectic:canonical} gives us a symplectic basis $\alpha_1^\pm, ..., \alpha_k^\pm, \beta_1^\pm, ..., \beta_k^\pm$ of $W(\Lambda_i)$ and a standard linear vector field $X\in \mathsf{sp}(W(\Lambda_i))$ with $\exp(X) = B_{\varphi^k}|_{W(\Lambda_i)}$. 
With similar notation as before, we compute

\begin{align*}
-\frac 12& \omega(X\sigma, \sigma) =\\
& \sum_{j = 1}^k  \log( |\Lambda_i|) \left(\ell_{\alpha_j^+}\ell_{\beta_j^+}+\ell_{\alpha_j^-}\ell_{\beta_j^-}\right) + \theta_i \left(\ell_{\alpha_j^-}\ell_{\beta_j^+} -\ell_{\alpha_j^+}\ell_{\beta_j^-}\right)  
 +\sum_{j = 1}^{k-1} \ell_{\alpha_{j+1}^+} \ell_{\beta_{j}^+}+ \ell_{\alpha_{j+1}^-} \ell_{\beta_{j}^-} . 
 \end{align*}
  
 \para{Case: unipotent blocks}
Let $W$ be the largest $B_{\varphi^k}$ invariant subspace of $\cH(\lambda_+)$ on which $B_{\varphi^k}$ is unipotent.
Then $W$ decomposes into a sum of minimal $B_{\varphi^k}$-invariant symplectic subspaces $W = U_1\oplus ... \oplus U_m$ and in each factor either $B_{\varphi^k}$ preserves a Lagrangian splitting or not (see \cite[Theorem 9]{Symplectic:canonical}).

For each minimal invariant subspace $V_i$, we find a symplectic basis $\alpha_1, ..., \alpha_k, \beta_1, ...,\beta_k \in \cH(\lambda_+)$ for $U_n$ and $X\in \mathsf{sp}(U_n)$ with $\exp(X) = B_{\varphi^k}|_{U_n}$.
If $B_{\varphi^k}$ preserves the Lagrangian splitting $U_n = \langle \alpha_1, ..., \alpha_k\rangle \oplus \langle \beta_1, ..., \beta_k\rangle$, then we use the normal form  \cite[(4)]{Symplectic:canonical} for $X$ to compute
\[-\frac12 \omega(X\sigma,\sigma) = \sum_{j = 1}^{k-1}\ell_{\alpha_{j+1}}\ell_{\beta_{j}}.\]

Otherwise, $B_{\varphi^k}$ preserves only $\langle \alpha_1, ..., \alpha_k\rangle$ and $k$ is even.  In this case we use \cite[(3)]{Symplectic:canonical} for the normal form of $X$ and compute
\[-\frac12 \omega(X\sigma,\sigma) = \frac{(-1)^{\frac k2}}{2}\ell_{\beta_k}^2+\sum_{j = 1}^{k-1}\ell_{\alpha_{j+1}}\ell_{\beta_{j}} .\]

  \para{Case: $\Lambda_i \not= 1 \in \Lambda_{\mathbb T}^+$}
  We proceed similarly as in the last case, where we let $W(\Lambda_i)$ be the largest $B_{\varphi^k}$ invariant subspace of $\cH(\lambda_+)$ corresponding to all the (real) Jordan blocks with with 
  \[L_i =\begin{pmatrix}
  \cos(\theta_i) & \sin(\theta_i)\\
  -\sin(\theta_i) & \cos(\theta_i)
  \end{pmatrix}• \]
along their main (block) diagonal.

Then $W(\Lambda_i)$ decomposes into a sum of minimal $B_{\varphi^k}$-invariant symplectic subspaces $W(\Lambda_i) = U_1(\Lambda_i)\oplus ... \oplus U_m(\Lambda_i)$ and again, each summand in this decomposition may or may not have an invariant Lagrangian splitting \cite[Theorem 14]{Symplectic:canonical}.

In case $B_{\varphi^k}$ preserves a Lagrangian splitting in $U_{n}(\Lambda_i)$, the dimension of $U_n(\Lambda_i)$ is divisible by $4$, and we find a symplectic basis $\alpha_1^+, ..., \alpha_k^+, \alpha_1^-, ..., \alpha_k^-, \beta_1^+, ..., \beta_k^+, \beta_1^-, ..., \beta_k^-$ for $U_n(\Lambda_i)$ in which $X$ is in its canonical form \cite[(9)]{Symplectic:canonical}.
We compute
\[-\frac12\omega(X\sigma, \sigma) =  \sum_{ j = 1} ^{k-1} \ell_{\alpha_{j+1}^+}\ell_{\beta_{j}^+} + \ell_{\alpha_j^-}\ell_{\beta_{j+1}^-} + \theta_i \sum_{j = 1}^k\ell_{\alpha_{k+1-j}^-}\ell_{\beta_j^+} - \ell_{\alpha_{k+1-j}^+}\ell_{\beta_j^-}.\]

There are two additional subcases when $B_{\varphi^k}$ preserves no Lagrangian splitting; that $k$ is even or odd. 
Either way there is a symplectic basis $\alpha_1, ..., \alpha_k, \beta_1, ..., \beta_k$ for $U_n(\Lambda_i)$ where $X$ is in the normal form \cite[(6)]{Symplectic:canonical} if $k$ is odd, and we have 
\[-\frac12\omega(X\sigma, \sigma) = \frac{\theta_i}{2}\sum_{j = 1}^{k}(-1)^{j+1}\left( \ell_{\alpha_j} \ell_{\alpha_{k+1-j}} + \ell_{\beta_j} \ell_{\beta_{k+1-j}} \right) + \sum_{j = 1}^{k-1} \ell_{\alpha_{j+1}}\ell_{\beta_j} \]

If $k$ is even, then $X$ is in its normal form \cite[(7)]{Symplectic:canonical} which allows us to compute
\[-\frac12\omega(X\sigma, \sigma) =- \frac{\theta_i}{2} \sum_{j = 1} ^{k}(\ell_{\alpha_j}\ell_{\alpha_{k+1-j}} + \ell_{\beta_j}\ell_{\beta_{k+1-j}}) + \frac12\sum_{j= 1}^{k-1}(-1)^j( \ell_{\alpha_{j+1}}\ell_{\alpha_{k+1-j}} + \ell_{\beta_j}\ell_{\beta_{k-j}}). \]

\section{Geometric train tracks}\label{sec:geom_tt}

The rest of the paper is devoted to finding an expression for the Poisson bracket between length functions for H\"older distributions. 
For this, we require estimates on the geometry of train track neighborhoods of geodesic laminations (Proposition \ref{prop:tt_geometry}).

\begin{note}[Constants]\label{note:constants}
All constants in the rest of the paper are understood to be local, and are often absorbed into ``big O'' notation.  It will be important however, that our `constants' may be considered as continuous functions of the relevant data so that we can make uniform estimates.  
\end{note}

\subsection{Horocyclic train track neighborhoods}
Recall that $\mathcal {GL}_\Delta(S)$ consists of minimal and filling geodesic laminations, and let $\lambda \in \mathcal {GL}_\Delta(S)$. 
Recall also that $\T^\lambda(S)$ is the locus of hyperbolic surfaces where the complement of $\lambda$ consists of regular ideal polygons, and let $Z\in \T^\lambda(S)$.

The following construction is due to Thurston \cite[Chapter 9]{Thurston:notes}.
For $\epsilon<1$, the \emph{$\epsilon$-horocyclic neighborhood} $\cN_\epsilon(\lambda)\subset Z$ is obtained by removing (open) segments of leaves of the horocycle foliation $H_\lambda(Z)$ (\S\ref{sub:horofoliation}) in the complementary polygons of $\lambda$ with length smaller than $\epsilon$. 
Note that $\cN_\epsilon(\lambda)$ is a closed set containing $\lambda$. 

If $H_\lambda(Z)$ has no closed leaves, then the leaf space $\tau_\epsilon = \cN_\epsilon(\lambda)/\sim$ of the $\epsilon$-horocyclic neighborhood has the structure of a train track.  We think of $\tau_\epsilon$ as being $C^1$ embedded in $\cN_\epsilon(\lambda)$.
The collapse map  $\cN_\epsilon(\lambda)\to \tau_\epsilon$ extends to a homotopy equivalence $Z\to Z$ homotopic to the identity on $Z$ witnessing $\lambda \prec \tau_\epsilon$.
Even if $H_\lambda(Z)$ does have closed leaves, if we take $\epsilon$ small enough, the collapse map will still extend to a homotopy equivalence of $Z$ homotopic to the identity; see Proposition \ref{prop:tt_geometry}.

Any train track $\tau_\epsilon = \tau(Z, \lambda, \epsilon)$ constructed from a triple of data  as above will be called a \emph{geometric train track} and comes equipped with its collapse map $\pi: \cN_\epsilon(\lambda)\to \tau_\epsilon$.  Minimality of $\lambda$ implies that we make take the parameter $\epsilon$ smaller if necessary to ensure that $\tau_\epsilon$ is trivalent.  
The \emph{ties} of $\tau_\epsilon$ or $\cN_\epsilon(\lambda)$ are connected components of restrictions of the leaves of $H_\lambda(Z)$ to the neighborhood; all the ties are $C^1$ and meet $\tau_\epsilon$ transversally. 
A train path $\gamma$ is \emph{induced by} or \emph{follows a leaf of $\lambda$} if there is a segment $g$ of a leaf $h\subset \lambda$ such that $\pi (g) = \gamma$. 

Given $k>0$, we may find $\epsilon>0$ depending only on $\inj(Z)$ and the topology of $S$ such that immersed train paths in $\tau(Z,\lambda,\epsilon)$ have average geodesic curvature at most $k$ \cite{CEG} on long enough segments. 

\subsection{Geometry of train track neighborhoods}\label{subsec:geometric_tt}
A train path $\gamma$ induced by a leaf of $\lambda$ has a well defined \emph{length} $\ell(\gamma)$, which is given by $\ell(g)$ for any segment $g\subset \pi\inverse(\gamma)$ of a leaf of $\lambda$. This notion of length is well defined as transporting a geodesic segment of a leaf of $\lambda$ along the the (orthogonal) horocycle foliation is length preserving.  The length $\ell(\tau)$ of $\tau$ is the sum of the lengths of the branches of $\tau$.

We will need a number of geometrical facts about train track neighborhoods, as constructed above; similar estimates are scattered about throughout the literature, e.g., in \cite{Brock:length} and \cite{CF:SHSHII}.  We supply proofs for completeness. 

\begin{proposition}\label{prop:tt_geometry}
Let $\lambda \in \mathcal {GL}_\Delta(S)$, suppose $Z\in \T^\lambda(S)$ is $\delta$-thick, i.e. $\inj(Z)\ge \delta$, and let $\epsilon<1$.
If $\cN_\epsilon(\lambda)$ has no closed ties, then $\tau_\epsilon =\tau(Z, \lambda, \epsilon)$ is a geometric train track with collapse map $\pi: \cN_\epsilon(\lambda)\to \tau_\epsilon$, and the following are true:
\begin{enumerate}
\item Every closed train path in $\tau_\epsilon$ has length at least $2\delta$.
\item The inequality $\ell(\tau_\epsilon) \le 6|\chi(S)|(\log1/\epsilon+d/2)$ holds, where $d$ is the maximum distance between adjacent horocycles of length $1$ in any polygonal component of $Z\setminus \lambda$.
\item There is a constant $E$ depending only on $\delta$ and the topology of $S$ such that  $\ell(t)\le E\epsilon$ for every tie $t\subset \cN_\epsilon (\lambda)$.
\item If $\epsilon <2E\inverse \delta$, then $\cN_\epsilon(\lambda)$ has no closed ties.
\item For $s<\epsilon$, any tie of $\cN_\epsilon(\lambda)$ meets at most $O(\log(\epsilon/s)+1)$ branches of $\tau_s$, counted with multiplicity.   
\end{enumerate}

\end{proposition}
\begin{proof}
The bound in item (1) follows from the fact that every closed curve carried by $\tau_\epsilon$ is homotopically essential, hence has length at least $2\inj(Z)\ge 2 \delta$.
Indeed, if a closed curve $\gamma$ carried by $\tau_\epsilon$ is not essential, then it  bounds an immersed disk in $Z$.
After adjusting by a homotopy, this disk is a union of polygonal components of $\tau_\epsilon$, and $\gamma$ runs along the boundary.
But the presence of sharp corners (spikes) on the boundary prevents the existence of a tangent map from $\gamma$ to $\tau_\epsilon$ induced by any $C^1$ carrying map.

We consider item (2). 
The length of a train path running from spike to spike along the boundary of a complementary polygon is equal to the distance along $\lambda$ between horocycles of length $\epsilon$.
A computation shows that this is at most $d +2\log1/\epsilon$.
Since every branch of $\tau_\epsilon$ has two sides, each polygon contributes at most half this quantity to the total length of $\tau_\epsilon$ for each of its interior edges.
The total number of interior edges is maximized when every polygon is a triangle.
Maximal laminations have zero area, hence there are $2|\chi(S)|$ triangles in the complement, each of area $\pi$. 
Stringing together these inequalities yields (2).

For item (3), we recall a theorem of Birman and Series \cite{BS}, which asserts that the Hausdorff dimension of the union over all Hausdorff limits of geodesic closed curves with bounded self intersection is zero.  In particular, for any transversal to a geodesic lamination, the intersection with $\lambda$ has one dimensional Lebesgue measure zero.  So the length of a tie is the sum of the lengths of its intersections with the spikes of $Z\setminus \lambda$.  There are at most $6|\chi(S)|$ spikes of $\lambda$, and each time a spike passes through $t$, it follows an essential train path in $\tau_\epsilon$ before returning to $t$.  Such a train path has length is at least $2\delta$, so the next horocyclic intersection of the spike with $t$ has length at most $e^{-2\delta}$ times the previous length.  The longest horocyclic arc through $t$ is less than $\epsilon$ in every spike, so we obtain 
\[\ell(t) \le 6|\chi(S)| \epsilon \sum_{r=0}^\infty e^{-r2\delta} =: E(\delta)\epsilon.\]

To prove (3), we observe that no leaf of $H_\lambda(Z)$ in the universal cover is closed, so closed ties of $\cN_{\epsilon}(\lambda)$ are homotopically essential. This means that their length is at least $2\delta$ and at most $E\epsilon$ by item (3).  Item (4) follows.

To prove the last item,  let $N$ be the number of intersections of $t\subset \cN_{\epsilon}(\lambda)$ with branches of $\tau_s$.  
The components of $t\setminus \tau_s$ correspond to spikes of $\lambda$ meeting $t$ in horocyclic arcs of length between $\epsilon$ and $s$.
There are at most $6|\chi(S)|$ spikes of $\lambda$, so the pigeon hole principal implies that there is a spike that cuts through $t$ at least $K = \lfloor \frac{N-1}{6|\chi(S)|}\rfloor$ times.
Before returning to $t$, this spike makes a non-trivial loop in $\tau_\epsilon$.
Thus the longest horocyclic arc of $t$ meeting this spike has length at least $se^{(K-1)2\delta}$.
On the other hand, this arc has length at most $\epsilon$.  
This gives 
\[K <\frac{1}{2\delta}\log(\epsilon/s) +1.\]

Remembering that $ \frac{N-1}{6|\chi(S)|}-1\le K$, we obtain 
\[N \le 6|\chi(S)|\left(\frac{1}{2\delta} \log(\epsilon/s) +2\right)+1 = O(\log (\epsilon/s)+1)\]
This concludes the proof of the proposition.    
\end{proof}

We will need the following estimate on the angle made by any pair of leaves of geodesic laminations which are close in some metric.

\begin{lemma}\label{lem:triple_angle_close}
Let $\lambda\in \mathcal {GL}_\Delta(S)$ and $Z\in \T^\lambda(S)$ be $\delta$-thick.  Suppose that $\lambda'$ is contained in an $\epsilon$-neighborhood of  $\lambda\subset Z$.  Let $k$ be a geodesic transversal to a branch $b$ of a geometric train track $\tau_\epsilon = \tau(Z,\lambda,\epsilon)$, and suppose that $k$ meets leaves $\ell\subset \lambda$ and $\ell'\subset \lambda'$.  
Denoting by $\theta$ and $\theta' \in (0, \pi)$ the angles made between $\ell$ and $k$ and $\ell'$ and $k$, respectively.  Then we have
\[|\cos(\theta) - \cos(\theta')| = O (\epsilon) ,\]
with implicit constants depending on $\delta$, the topology of $S$, and the angle made by $k$ with any leaf of $\lambda$.  
\end{lemma}

\begin{proof}
Let us consider the function on $k\cap \lambda$ measuring the angle made between leaves of $\lambda$ and $k$.
By the universal Lipschitz property of geodesic laminations (Lemma \ref{lem:lipschitz}), this  function is locally $L$-Lipschitz.
So the angle made by any leaf of $\lambda$ with $k$ is at most $L$ times the diameter of the set $k\cap \lambda$, which is contained in the foliated band of $\cN_{\epsilon}(\lambda)$ collapsing to the branch $b$.
By Proposition \ref{prop:tt_geometry}, the length of any tie foliated this band is at most $E\epsilon$. 
Then the diameter of $k\cap \lambda$ is bounded by $AE\epsilon$, with $A$ depending on the angle made between $k$ and some leaf of $\lambda$.
We have shown that for any points $p$ and $p'$ of $k\cap \lambda$, the angles made between $k$ and $\lambda$ at $p$ and $p'$ differ by at most $LAE\epsilon=O(\epsilon)$.

If we can show also that some leaf of $\lambda$ is $O(\epsilon)$ close in the unit tangent bundle to some leaf of $\lambda'$ at their intersections with $k$, then the lemma will follow, appealing to the fact that cosine is Lipschitz.
Give $k$ an orientation and let $\ell_0\subset \lambda$ and $\ell_0'\subset \lambda'$ be the first geodesics that $k$ crosses.
Then $\ell_0$ and $\ell_0'$ bound polygons $P$ and $P'$ complementary to $\lambda$ and $\lambda'$, hence induce the same train path of length at least $2\log(1/\epsilon)$ running between two spikes of $\tau_\epsilon$.
Proposition \ref{prop:tt_geometry} implies that $\ell_0$ and $\ell_0'$ fellow travel at distance $(E+2)\epsilon$ along this path of length at least $2\log(1/\epsilon)$.
This is enough to show that $\ell_0$ and $\ell_0'$ are $O(\epsilon)$ close in the unit tangent bundle at their intersections with $k$, and concludes the proof of the lemma.
\end{proof}

\subsection{Growth}
We consider an oriented geodesic transversal $k\subset Z$ meeting a single branch $b\subset \tau_\epsilon$. The orientation of $k$ defines a linear order on the geodesics that meet $k$, and we locally give orientations that form a positively oriented basis for the tangent space of $Z$ at every point of intersection; the orientation can be made (non-equivariantly) global when working in universal covers.

For a component $d\subset k\setminus \lambda$, $d^+$ is the positive endpoint of $d$ meeting the leaf $h_d^+$ of $\lambda$ and $d^-$ is the negative endpoint meeting the leaf $h_d^-$; note that $h_d^-$ and $h_d^+$ are asymptotic and form an \emph{oriented spike} $s_d = (h_d^-, h_d^+)$, where the orientation is induced by that of $k$.  
The bi-infinite train paths induced by $h_d^-$ and $h_d^+$ fellow travel along a ray in one direction and eventually diverge in the other.  The \emph{divergence radius} $r(d)\ge 1$ is the number of branches that train paths following $h_d^-$ and $h_d^+$ agree on in their non-asymptotic direction.\footnote{The divergence radius $r(d)$ depends on the train track $\tau_\epsilon$ and branch $b$ that $k$ crosses, but we omit these data from the notation.}
There are also points $k^-$ and $k^+ \subset k\cap \lambda$ meeting leaves $h^-$ and $h^+$ closest to the start and end of $k$, respectively.
Choosing a point in $d$, we let $k_d$ be the subsegment of $k$ joining the negative endpoint of $k$ to that point.

We will need some bounds on the growth of the size of a H\"older distribution $\alpha \in \cH(\lambda)$ evaluated on $k_d$.
Recall from \S\ref{sub:THD} that there is a  canonical integer linear isomorphism between $W(\tau_\epsilon)\le \RR^{b(\tau_\epsilon)}$ and $\cH(\lambda)$. 
We define a norm $\| \cdot \|$ on $\cH(\lambda)$ as the restriction of the $\ell^\infty$ norm on $\RR^{b(\tau_\epsilon)}$ to $W(\tau_\epsilon)$.  
The following facts are extremely useful when working with shear coordinates; see Lemmas 3, 4, 5, and 6 of \cite{Bon_SPB}, Lemmas 4, 5, and 6 of \cite{BonSoz}, or Lemmas 14.2, 14.3, and 14.5 of \cite{CF}.
\begin{lemma}\label{lem:depth}
With notation as above and $\alpha\in \cH(\lambda)$, we have a constant $A$ depending on the angles of intersection of $k$ with $\lambda$ such that
\begin{enumerate}
\item $|\alpha(k_d)| \le \| \alpha\| r(d)$; and 
\item $\ell(d) \le  A \epsilon\exp\left(\frac{-r(d) 2\inj(Z)}{9|\chi(S)|}\right)$; and
\item For each $r\ge 1$ and each spike $s$ of $\lambda$, there is at most one component $d$ of $k\setminus\lambda$ contained in $s$ such that $r(d) = r$, and there are $6|\chi(S)|$ spikes of $\lambda$. 
\end{enumerate}
\end{lemma}
Note that we essentially proved items (2) and (3) in the course of proving Proposition \ref{prop:tt_geometry}.

We can relate the weight deposited on a branch of a geometric train track $\tau_\epsilon = \tau(Z,\lambda, \epsilon)$ with its weights on $\tau_{\epsilon_0} = \tau(Z, \lambda, \epsilon_0)$, for $\epsilon\ll \epsilon_0$.  
For the next Lemma, let $\|\cdot\|$ denote the restriction of the $\ell^\infty$-norm on $\RR^{b(\tau_{\epsilon_0})}$ to $W(\tau_{\epsilon_0})$.  

\begin{lemma}\label{lem:norm}
Let $t_b$ be a tie for a branch of $\tau_\epsilon$ and $\alpha\in \cH(\lambda)$.    Then for $\epsilon$ small enough compared to $\epsilon_0$,  $|\alpha(t_b)| \le C \|\alpha\| \log(\epsilon_0/\epsilon)$ holds,  
with $C$ depending on $\inj(Z)$, and the topology of $S$.
\end{lemma}
\begin{proof}
There is a branch $b_0$ of $\tau_0$ onto which the tie $t_b$ collapses.  Let $t_0$ be a tie of $\tau_0$ containing $t_b$, and give $t_0$ an orientation.  Choose points $x_-$ and $x_+$ in $t_0\setminus \cN_\epsilon(\lambda)$ adjacent to $t_b$ on its negative and positive sides, respectively.  Take a subarc $t_-$ joining the negative endpoint of $t_0$ to $x_-$ and $t_+$ joining the positive endpoint of $t_0$ to $x_+$.  By transverse invariance and finite additivity, we have 
\[ \alpha( t_b) = \alpha(t_-) + \alpha(t_+) -\alpha(t_0).\] 

According to Lemma \ref{lem:depth} item (2), $\epsilon \le \epsilon_0\exp \left(\frac{ -r(t_+)2\inj(Z)}{9|\chi(S)|}\right)$\footnote{We take $A$ from Lemma \ref{lem:depth} item (2) to be equal to $1$ here, since we are working with the lengths of horocyclic arcs.}, which in particular implies that 
\[r(t_+)\le \frac{9|\chi(S)|}{2\inj(Z)}\log(\epsilon_0/\epsilon),\] 
and similarly that $r(t_-) \le \frac{9|\chi(S)|}{2\inj(Z)}\log(\epsilon_0/\epsilon)$.
Then by Lemma \ref{lem:depth} item (1), we deduce that 
\begin{equation}\label{eqn:cocycle_bounded}
|\alpha(t_b)| \le \|\alpha\| \left( 1 + 2 \frac{9|\chi(S)|}{2\inj(Z)}\log(\epsilon_0/\epsilon)\right) \le C\|\alpha\|\log(\epsilon_0/\epsilon),
\end{equation}
demonstrating the lemma for $\epsilon$ small enough with $C>\frac{9|\chi(S)|}{\inj(Z)}$ large enough to absorb the additive error. 
\end{proof}

\section{Dynamics of the stretch flow}\label{sec:stretch_flow}
In the next section, we will obtain estimates on the rate of Hausdorff convergence of supports for certain $C^1$ paths of measured laminations (Theorem \ref{thm:measure_H_close}).
The argument is dynamical in nature, and our discussion relies on a dictionary between surface-lamination pairs and half-translation structures, i.e., holomorphic quadratic differentials.

In this section, we discuss some preliminaries on dynamics of the earthquake and stretch flows on the moduli space of hyperbolic surfaces equipped with a geodesic lamination, as well as on the structure of the invariant measures corresponding to components $\cQ$ of strata of quadratic differentials.
The main result from this section is Corollary \ref{cor:recurrent}, which states that almost every minimal filling measured geodesic lamination of topological type coming from $\cQ$ has a certain recurrence property with respect the to $\cQ$-Thurston measure on $\ML(S)$.

\subsection{Holomorphic quadratic differentials}
Let $Z\in \T(S)$.
    A \emph{holomorphic quadratic differential} $q\in \Omega^{2,0}(Z)$ is a holomorphic section of the symmetric square of the holomorphic cotangent bundle. 
    Path integration against a branch of the square root  of $q$ produces holomorphic charts to $\CC$ where the transitions are translations and translations with $\pi$-rotation.
    The transition maps preserve the Euclidean metric and directional foliations in $\CC$, so $q$ defines a flat metric on $Z$ away from the zeros of $q$, which we equip with its imaginary $(\cH, |dy|)$ and real $(\mathcal V, |dx|)$ measured foliations. 
    The metric completion of this flat structure has cone points of excess angle $\pi k$ at each zero of order $k$ where there are  $k$-pronged singularities of both $\cH$ and $\mathcal V$.  A \emph{saddle connection} is a (singularity free) Euclidean segment joining zeros of $q$.

	For each $q \in \Omega^{2,0}(X)$, there is a constant $\epsilon>0$ such that for all essential simple closed curves $\gamma$ 	\begin{equation}\label{eqn:binding}
	     i((\mathcal V, |dx|), \gamma) + i((\cH, |dy|),\gamma)>\epsilon.
	\end{equation}
	We call a pair of meausred foliations \emph{binding} if they satisfy \eqref{eqn:binding}.
	It turns out that this property  characterizes those pairs of measured foliations that can be realized as the real and imaginary foliations of a holomorphic quadratic differential.  
	Indeed, let $\Delta$ be the subset of $\MF(S) \times \MF(S)$ that do not bind, and let $\mathcal {QT}(S)$ be the complex vector bundle of holomorphic quadratic differentials over $\T(S)$.
	\begin{theorem}[{\cite{GM}}]\label{thm:GM}
	    The map that assigns to a holomorphic quadratic differential the real and imaginary measured foliations is a $\operatorname{Mod}(S)$-equivariant homeomorphism.  That is, 
	    \[\mathcal {QT}(S) \cong \MF(S)\times \MF(S) \setminus \Delta.\]
	\end{theorem}

The moduli space of quadratic differentials $\QM(S) = \QT(S)/\Mod(S)$ is stratified by specifying the number and multiplicities of zeros.
Thus for a partition $\kappa$ of $4g-4$, there is a corresponding  stratum of quadratic differentials $\mathcal Q(\kappa) \subset \QT(S)$.  
There are finitely many components in each stratum, classified by topological data (except for some exceptional cases in genus $4$).
Each stratum is a complex orbifold; there are natural local period coordinates on each component of each stratum obtained by specifying a basis for a certain relative cohomology group recording the periods.

There is a natural action of $\mathsf{GL}_2\RR$ on $\QT(S)$ preserving the strata given by postcomposing the natural half-translation charts with  a given linear map of $\mathbb C= \RR^2$.
The $1$-parameter groups 
\[g_t = \begin{pmatrix}e^t & 0 \\ 0 & e^{-t}\end{pmatrix} \text{ and } u_t =  \begin{pmatrix}1 & t \\ 0 & 1\end{pmatrix}\] correspond to the \emph{Teichm\"uller geodesic} and \emph{horocycle} flows, respectively.  
There is a $\mathsf{GL}_2\RR$-invariant \emph{Masur-Veech} probability measure $\mu_{\mathcal Q}$ in the class of Lebesgue on the unit area locus of each component $\mathcal Q$ of each stratum.
Both $g_t$ and $u_t$ are ergodic with respect to $\mu_{\mathcal Q}$  \cite{Masur_IETsMF, Masur:ergodic_MCG, Veech:ergodic}.

\subsection{The dynamical conjugacy}
Let $\kappa=(\kappa_1, ..., \kappa_j)$ be a partition of $4g-4$ and consider the locus $\ML(\kappa)$ of measured geodesic laminations with minimal and filling support whose complement consists of $j$ polygons with $\kappa_{i}+2$ sides.
We consider the ``stratum'' of regular pairs \[\PT(\kappa) = \bigcup _{\mu\in \ML(\kappa)}\{(Z,\mu): Z\in  \T^\mu(S)\},\] with quotient $\PM(\kappa) = \PT(\kappa)/\Mod(S)$.  The unit length locus is denoted $\PoM(\kappa)= \PoT(\kappa)/\Mod(S)$. 
Using the horocycle foliation construction, we can define a map 
\[H : \bigcup_\kappa\PT(\kappa) \to \QT(S)\]
where $H(Z,\mu)$ is the unique quadratic differential $q(H_\mu(Z), \mu)$ with real foliation $H_\mu(Z)$ and imaginary foliation equivalent to $\mu$.
For $\kappa = (1, ..., 1)$, the following is due to Mirzakhani, while for arbitrary $\kappa$, we use the results of Calderon-Farre.
\begin{theorem}[ \cite{MirzEQ,CF,CF:SHSHII}]\label{thm:conjugacy}
For each $\kappa$, the map $H$ restricts to a Borel-Borel $\Mod(S)$-equivariant bijection between $\PT(\kappa)$ and the locus of $\QT(\kappa)$ with no horizontal saddle connections with the following additional properties
\begin{itemize}
\item $\ell_\mu(Z) = \Area(H(Z,\mu))$. 
\item The earthquake flow is mapped to the Teichm\"uller horocycle flow, i.e., \[H(\Eq_{t\mu}(Z),\mu) = u_t. H(Z,\mu).\]
\item The generalized stretch flow is mapped to the Teichm\"uller geodesic flow, i.e., \[H(\operatorname{stretch}(Z,\mu,t),e^{-t}\mu) = g_t.H(Z,\mu).\]
\end{itemize}
\end{theorem}

\begin{remark}
The map $H$ is the restriction of an everywhere defined $\Mod(S)$-equivariant Borel-Borel bijection $\cO:\PT(S) \to \QT(S)$ using the \emph{orthogeodesic foliation construction}, which is (measure) equivalent to the horocycle foliation construction for polygonal, regular surface lamination pairs \cite{CF,CF:SHSHII}.
\end{remark}

The locus of $\QM(\kappa)$ with no horizontal saddle connections has zero measure for any $\mathsf{GL}_2(\RR)$-invariant measure, so $H$ has a measurable inverse on each stratum with respect to any such measure.
In particular, for each component $\cQ$ of $\QM(\kappa)$, we have a measurable inverse $H\inverse: \cQ \to \PM(\kappa)$ pushing the Masur-Veech measure $\mu_{\cQ}$ to a stretchquake  (earthquake and stretch flow)  invariant Borel  measure $\nu_{\cQ}=H\inverse_*\mu_{\cQ}$ on $\PM(\kappa)$, which induces a probability measure with the same name on $\PoM(\kappa)$.

The following is an immediate consequence of Theorem \ref{thm:conjugacy} and the corresponding facts that the Teichm\"uller flows are ergodic on the unit area loci for every component of every stratum of quadratic differentials.
\begin{corollary}\label{cor:stretchquake_ergodic}
The earthquake flow and generalized stretch flow are both ergodic on the unit length locus with respect to $\nu_{\cQ}$ for every component of every stratum. 
\end{corollary}

\subsection{$\cQ$-Thurston measures}\label{sub:Thurston}
With $\kappa = (1, ..., 1)$, Mirzakhani proved for the principal (connected, open) stratum $\cQ = \cQ(\kappa)$, that $\nu_{\cQ}$ is locally the product of the Weil-Petersson volume form with the Thurston measure on measured laminations \cite{MirzEQ}.  
The Thurston measure $\mu_{\Th}$ on $\ML(S)$ is constructed as the Lebesgue measure in maximal bi-recurrent train track charts, normalized so that the integer lattice has co-volume $1$ in each chart.  As the transitions between train track charts are piecewise integer-linear and invertible, these Lebesgue measures piece together into a globally defined $\Mod(S)$-invariant (ergodic) measure on $\ML(S)$; see, e.g., \cite{PennerHarer}.  

We outline the construction for the \emph{$\cQ$-Thurston measure} $\mu_{\Th}^\cQ$ on $\ML(S)$ and refer the reader to \cite{CF} for details.
Apart from some sporadic cases in genus four, the stratum $\cQ$ can be identified in terms of topological properties of the horizontal foliation of the $\mu_\cQ$-typical quadratic differential (those without horizontal saddle connections).
We let $\mathscr T(\cQ)$ be the set of isotopy classes of  bi-recurrent train tracks of type $\kappa$ on $S$ subject to the additional topological restrictions (such as orientability, hyperellipticity, etc.) imposed by the horizontal foliation of the $\mu_{\cQ}$-typical differential.  
Note that $\mathscr T(\cQ)$ is  $\Mod(S)$-invariant.

Define
\[\ML(\cQ) = \bigcup_{\tau \in \mathscr T(\cQ)} U_{\operatorname{snug}}(\tau),\]
where $U_{\operatorname{snug}}(\tau)\subset \ML(S)$ is the cone of measures snugly carried by $\tau$.\footnote{$U_{\operatorname{snug}}(\tau)$ is the complement of countably many integer-linear hyperplanes, and is therefore a Borel set.}
Then the $\mu_{\Th}^\cQ$ is defined in each cone $U_{\operatorname{snug}}(\tau)$ as the Lebesgue measure, normalized so that the integer lattice\footnote{the conditions cutting out $U_{>0}(\tau)$ from $\ML(S)$ are piecewise $\ZZ$-linear, so the integer lattice in $\ML(S)$ restricts to an integer lattice in $U_{>0}(\tau)$.} has co-volume $1$.
Again, the transitions between charts are piecewise integer-linear, so these measures can be patched together to form a $\Mod(S)$-invariant measure $\mu_{\Th}^\cQ$ on $\ML(S)$; then $\ML(S)\setminus \ML(\cQ)$ has $\mu_{\Th}^\cQ$-measure $0$.

The usual Thurston measure induces an inner and outer regular Borel probability measure on the unit length locus $\ML^1(Z)$ for any $Z\in \T(S)$, but the $\cQ$-Thurston measure is not outer regular or even locally finite on $\ML^1(Z)$, unless $\cQ$ is the principal stratum; see \cite{CF} and \cite{LM}.

Nevertheless, we have the following structural result regarding $\nu_{\cQ}$ generalizing Mirzakhani's disintegration formula.  

\begin{theorem}[\cite{CF}]\label{thm:measures}
Let $\kappa$ be a partition of $4g-4$, and $\cQ\subset \QM(\kappa)$ be a component.  The measure $\nu_{\cQ}$ on $\PM(\kappa)$ is locally the product of the $\cQ$-Thurston measure and a volume form on the regular locus $\T^\mu(S)$ induced by the (degenerate) Weil-Petersson symplectic form tangent to $\T^\mu(S)$.  
\end{theorem}

\begin{remark}
Part of the content of Theorem \ref{thm:measures} is the definition of the volume forms on regular loci and the construction of  (locally defined) volume-preserving diffeomorphisms between open sets in $\T^\mu(S)$ and $\T^{\mu'}(S)$ if $\mu$ and $\mu'$ are close enough in measure.  See \cite{CF} for a precise formulation and further discussion.
\end{remark}

We make use of the following corollary for arbitrary components $\cQ\subset \QM(\kappa)$.

\begin{corollary}\label{cor:recurrent}
Let $\cR(\cQ)\subset \ML(\cQ)$ be the set of measured geodesic laminations of type $\cQ$ satisfying the following property: for any $\mu \in \cR(\cQ)$ and any $Z\in\T^\mu(S)$, the projection of the anti-stretch path $(\operatorname{stretch}(Z, \mu, -t) , e^t\mu)\in \PT(\kappa)$ to $\PM(\kappa)$ returns to some compact set in $\PM(\kappa)$ infinitely often as $t\to \infty$.

Then $\cR(\cQ)$ has full $\mu_{\Th}^{\cQ}$-measure, i.e.,  $\mu_{\Th}^{\cQ}(\ML(\cQ)\setminus \cR(\cQ)) = 0$.  
\end{corollary}

\begin{proof}
By corollary \ref{cor:stretchquake_ergodic}, the (generalized) stretch flow is ergodic for $\nu_{\cQ}$ on the locus of constant length pairs in $\PM(\cQ)$.  
Thus $\nu_{\cQ}$-a.e. pair is recurrent in backward time.  
By Theorem \ref{thm:measures}, $\nu_{\cQ}$ is locally the product of the $\cQ$-Thurston measure and a volume element on the regular locus in $\T(S)$.  
An application of Fubini's Theorem completes the proof.
\end{proof}

\section{Measure vs. Hausdorff convergence}\label{sec:MH_convergence}

Our main result in this section (Theorem \ref{thm:measure_H_close}) relates measure convergence of certain $C^1$ paths of measured laminations with the Hausdorff convergence  of their supports along a subsequence, and may be of independent interest.
It appears that a better rate of convergence could be obtained by studying the Lyapunov exponents and Oseledets flag for (a hyperbolic geometric version of) the Kontsevich-Zorich cocycle.
However, we do not pursue this here, as the naive rate of convergence obtained below suffices for our purposes.

As we split a geometric train track for a pair $(Z,\mu) \in \PoT(\kappa)$ by taking a small parameter to zero, the measures deposited on the thin split track by $\mu$ are very small. The next proposition states that for $\nu_\cQ$ almost every pair, there is an infinite collection  of ``stopping times," where the measure deposited on the corresponding tracks is ``balanced.''

\begin{proposition}\label{prop:recurrence}
Let $\cQ$ be a component of a stratum $\QM(\kappa)$ of quadratic differentials, and $\nu_\cQ$ be the corresponding stretchquake invariant ergodic probability measure on the unit length locus $\PoM(\kappa)$.
For $\nu_\cQ$-almost every pair $(Z,\mu) \in \PoM(\kappa)$, there are constants $M$ and $t_0$ and a sequence $t_1,t_2, ...$ tending to infninty, such that the weights of $\mu$ deposited on the branches of $\tau_n = \tau(Z, \mu, e^{-e^{t_n}t_0})$ are bounded between $\frac{1}{M}e^{-t_n}$ and $Me^{-t_n}$ for all $n$.  
\end{proposition}

\begin{proof}
Let $\cR^-(\cQ)\subset \PoM(\kappa)$ be the set of pairs which recur to compact sets in backward time under the stretch flow on $\PoM(\kappa)$.  
By ergodicity (Corollary \ref{cor:stretchquake_ergodic}), $\cR^-(\cQ)$ has full $\nu_{\cQ}$ measure.
Let $\widetilde\cR^-(\cQ)\subset \PoT(\kappa)$ be the full preimage under the $\Mod(S)$ orbit projection.

Let $(Z,\mu)\in \widetilde \cR^-(\cQ)$ and let  $t_0$ be small enough so that the geometric train track $\tau_0 = \tau(Z,\mu, e^{-t_0})$ snugly carries $\mu$.  
By taking $t_0$ smaller if necessary, we may also assume that $\tau_0$ is generic/trivalent.
Since $\tau_0$ is generic, it is structurally stable, i.e., if $(Z', \mu')$ is close enough to $(Z,\mu)$ in the sense that the metric structures of $Z$ and $Z'$ are close and the supports of $\mu$ and $\mu'$ are Hausdorff close on both metrics, then $ \tau(Z', \mu', e^{-t_0})$ is isotopic to $\tau_0$.  
The set of pairs $\mathcal K\subset \PoT(\kappa)$ near $(Z,\mu)$ satisfying this condition  has positive $\nu_\cQ$-measure; making $\mathcal K$ smaller if necessary, we also require that the closure of the projection of $\mathcal K$ to $\ML(S)$ stays in the interior of the positive cone of measures carried by $\tau_0$.
Then there is a constant $M$ such that
\[ \frac{1}{M}\le \mu'(t_b)\le M\]
for any transversal/tie $t_b$ to any branch $b$ of $\tau_0$ and for all $\mu'$ with $(Z',\mu')\in \mathcal K$.

Now, let $t_1, t_2, ...$ be an unbounded increasing sequence of recurrence times to the projection of $\mathcal K$ in the moduli space $\PoM(\kappa)$ under the anti-stretch flow.  
That is, we have mapping classes $\varphi_n\in \Mod(S)$ such that
\[\varphi_n.(\operatorname{stretch}(Z,\mu,-t_n), e^{t_n}\mu)\in\mathcal K. \]

The time $s$ anti-stretch mapping is in the homotopy class of the identity and  takes horocycles of length $e^{-t}$ in the spikes of $\mu$ to horocycles of length $e^{-e^{-s}t}$ in the spikes of $\mu$.
Thus 
\[\tau_n=\tau(Z,\mu,e^{-e^{t_n}t_0}) \text{ is isotopic to } \tau(\operatorname{stretch}(Z,\mu,-t_n), \mu, e^{-t_0}).\] 
We set $Z_n = \varphi_n.\operatorname{stretch}(Z,\mu,-t_n)$ and $\mu_n = \varphi_n.e^{-t}\mu$. By definition of $\mathcal K$, we have that $\tau(Z_n,\mu_n,e^{-t_0})$ is isotopic to $\tau_0$.
In particular, $\varphi_n.\tau_n$ is isotopic to $\tau_0$, and $\mu_n$ lies in the interior of the positive cone in the weight space of $\tau_0$.

By our choice of $M$, we know 
\[\frac{1}{M}\le \mu_n(t_b) \le M,\]
for any transversal $t_b$ to any branch $b$ of $\tau_0$.
Equivalently, 
\[ \frac{e^{-t_n}}{M}\le  \mu(t_{b'}) \le Me^{-t_n}\]
for any transversal $t_{b'}$ to any branch $b'$ of $\tau_n$, and this is what we wanted to show.
\end{proof}

We use the previous collection of balanced train tracks converging to $\mu$ to control the rate of geometric convergence of (germs of) $C^1$ paths of measured laminations.
\begin{theorem}\label{thm:measure_H_close}
For $\mu_{\Th}^\cQ$-almost every $\mu\in \ML(\cQ)$ and any chain recurrent diagonal extension $\lambda$,  for every $\alpha\in \cH(\lambda)$ representing a tangent vector to $\mu$ in $\ML(S)$, and every $Z\in \T(S)$, there are constants $c_1, c_2 \ge 1$, $a\in (0,1]$, and $t_0 >0$ depending continuously on $Z$, $\mu$, and the size of $\alpha$ such that the following holds. 

There is a sequence $t_1, t_2, ... $ tending to infinity such that if $s\le c_1e^{-2t_n}$, then a $c_2\exp\left({-at_0e^{t_n}}\right)$ neighborhood of  $\mu_s:= \mu+s\alpha$ on $Z$ contains  $\mu$. 
\end{theorem}

\begin{proof}
Assuming that $\mu$ is in the $\mu_{\Th}^\cQ$-full measure set  $\cR({\cQ})$ from Corollary \ref{cor:recurrent}, we can find a surface $Z' \in \T^\mu(S)$ such that $H(Z',\mu) \in \widetilde \cR^-(\cQ)$.
We apply Proposition \ref{prop:recurrence} to obtain constants $M$, $t_0$, and a sequence $t_1, t_2, ...$ tending to infinity so that for the family of geometric train tracks $\tau_n = \tau(Z', \mu, e^{-t_0e^{t_n}})$, we have 
\[\frac {e^{-t_n}}M \le \mu(t_b) \le e^{-t_n}M, \]
for any transversal $t_b$ to any branch $b$ of $\tau_n$.

Let $\tau_n'$ be a diagonal extension of $\tau_n$ corresponding to the isolated leaves of $\lambda$.
Then $\tau_n'$ carries $\mu_s$ if, for any branch $b$ and a transversal $t_b$ to that branch, we have $\mu_s(t_b) \ge 0$.
By Theorem \ref{thm:tangent}, since $\alpha$ represents a tangent vector to $\mu$, $\mu_s$ is positive on any branch of $\tau_n'$ corresponding to an isolated leaf of $\lambda$.

We will determine how small $s$ must be in order for $\mu_s$ to be carried by $\tau_n'$; let $b$ be a branch of $\tau_n'$ that does not come from an isolated leaf of $\lambda$, and let $t_b$ be a transversal.
Using Lemma \ref{lem:norm}, we have 
\[\mu(t_b) - s |\alpha(t_b)|\ge \frac{e^{-t_n}}{M} - sC\|\alpha\|e^{t_n},\]
where $\| \alpha\|$ is the $\ell^\infty$ norm of $\alpha$ as a weight system on $\tau_0 = \tau(Z',\mu, e^{-t_0})$ and $C = 9|\chi(S)|/\inj(Z')$.
This is non-negative if $s\le \frac{e^{-2t_n}}{MC\|\alpha\|} = c_1e^{-2t_n}$.  Thus if $s$ is at least this small,  $\mu_s$ must be carried by $\tau_n'$.

Adjusting $\tau_n$ by isotopy if necessary, the average geodesic curvature of long train paths in $\tau_n$ can be made arbitrarily small as $n\to \infty$.  It follows that there are scalars $k_1, k_2, ...$ decreasing to $1$, such that $\mu_s$, straightened in $Z'$, lies in the $k_n\exp\left(-t_0 e^{t_n}\right)$-neighborhood of $\tau_n'$ as long as $s\le c_1e^{-2t_n}$.  Furthermore, $\mu$ is carried by $\tau_n'$, so its straightening is also contained in the $k_n\exp\left(-t_0 e^{t_n}\right)$-neighborhood of $\tau_n\subset \tau_n'$.  

Let $\mu\subset\lambda'\subset \lambda$ be the diagonal extension of $\mu$ on which $\alpha$ is supported. 
Let $d_{Z'}^H$ denote the Hausdorff distance on closed subsets with respect to the hyperbolic structure $Z'$.
 We have proved that 
 \[d_{Z'}^H(\mu_s,\lambda')\le 2k_n\exp\left(-t_0 e^{t_n}\right) \text{ as long as } s\le c_1e^{-t_n}.\]
The Hausdorff metrics on the space of geodesic laminations are all H\"older equivalent to one another \cite{BonZhu:HD}, so there are constants $c\ge 1$ and $a\in (0, 1]$ depending on $Z$ and $Z'$ 
such that 
\[d_{Z}^H(\mu_s, \lambda') \le cd_{Z'}^H(\mu_s,\lambda')^a. \]
Collecting constants appropriately completes the proof.
\end{proof}

\section{Product Distributions and shearing cocycles}\label{sec:int}
In this section, we make sense of the double integrals with respect to the product distribution $\alpha\otimes \beta$ evaluated on nice enough functions.  
We observe that Fubini's Theorem holds for H\"older distributions, thought of as H\"older currents on the space of geodesics of $S$ for functions that are in the closure of simple tensors (Lemma \ref{lem:Fubini}).  
Then the geometric intersection form is extended to transverse H\"older cocycles, and we explain its meaning as the second derivative of the usual intersection form on $\ML(S)$ (Lemma \ref{lem:intersection}).

Next, we establish some basic estimates, most of which can be extracted from \cite{Bon_THDGL, Bon_GLTHB, Bon_SPB}.    And show that certain ``Riemann sums'' can be used to compute double integrals of H\"older functions for which we can apply Fubini's Theorem (Proposition \ref{prop:R_sums_prod}).  The proof of this proposition is more technical than anticipated; it is also  more general than what we require in Section \ref{sec:cosine} to establish Theorem \ref{thm:cosine_intro}, where we only consider \emph{smooth} functions.

We conclude with a description of Bonahon's shearing cocyles associated to a H\"older cocycle on a maximal geodesic lamination.  These shearing cocyles correspond to (germs of) analytic paths in $\T(S)$. 
\subsection{Currents and intersection}

A \emph{geodesic H\"older current} is a H\"older distribution on the space $\mathcal G(\widetilde S)$ of geodesics in the universal cover which is invariant under the action of $\pi_1(S)$ by covering transformations.  The space of geodesics has a well defined bi-H\"older structure, independent of negatively curved metric on $S$.  Given $\lambda_i \in \mathcal {GL}_0(S)$ and H\"older cocyles $\alpha_i\in \cH(\lambda_i)$ ($i = 1,2$), we can associate uniquely a geodesic H\"older current to $\alpha_i$ whose support is contained in the leaves of $\widetilde{\lambda_i}$ \cite[Proposition 5]{Bon_THDGL}.  
We  form the tensor product H\"older distribution $\alpha_1\otimes \alpha_2$ on the product $\mathcal {G}(\widetilde S)\times \mathcal {G}(\widetilde S)$ which is again invariant by covering transformations. 

For any H\"older continuous function $f:\mathcal {G}(\widetilde S)\times \mathcal {G}(\widetilde S) \to \RR$, we have  $\alpha_1\otimes\alpha_2(f) = \alpha_1\otimes \alpha_2(f|_{\widetilde{\lambda_1}\pitchfork\widetilde{\lambda_2}})$, where  $\widetilde{\lambda_1}\pitchfork\widetilde{\lambda_2}$ denotes the subset $(h_1, h_2)\in \widetilde{\lambda_1}\times \widetilde{\lambda_2}$  such that the endpoints of $h_1$ are interlaced with the endpoints of $h_2$ in the circle at infinity, i.e. $h_1$ meets $h_2$ transversally in $\HH^2$.  

We define $\mathcal {DG}(\widetilde S)$ as the subspace of $\mathcal {G}(\widetilde S)\times \mathcal {G}(\widetilde S)$ consisting of geodesics meeting transversally and denote by $\mathcal {DG}(S)$ the quotient by the diagonal action of $\pi_1(S)$, which is free and properly discontinuous.  
A choice of hyperbolic metric $Z\in \T(S)$ on $S$ identifies $\mathcal {DG}(S)$ homomorphically with the total space of the bundle  $\mathbb PT(S)\oplus\mathbb PT(S)\setminus \Delta$ equipped with a bundle projection to $S$.  We can endow $\mathcal {DG}(S)$ with an isotopy class of  Riemannian metric $m_Z$ such that the right action of a maximal compact $K\le \Isom^+(\tX)$ on each factor of the fiber over a point of $Z$, is isometric. We let $\mathcal {DG}(Z)$ denote $\mathcal {DG}(S)$ equipped with the isotopy class of $m_Z$.

The following Fubini-type theorem for distributions on $\mathcal {DG}(S)$ follows from the analogous result for distributions on $\RR^4$.  
\begin{lemma}[c.f. Theorem 40.4 of \cite{Treves}]\label{lem:Fubini}
Given $Z\in \T(S)$, suppose that $f \in \cH^a(\mathcal {DG}(Z))$ is in the closure 
of the linear span of simple tensors $\cH^a(\mathcal G(Z))\otimes \mathcal \cH^a(\mathcal G(Z))$ (with respect to the $a$-H\"older norm), then
\begin{equation}\label{eqn:Fubini}
\alpha_1 \otimes \alpha_2 (f) = \int \left( \int f(x,y) ~d\alpha_1(x)\right) ~d\alpha_2(y) = \int \left( \int f(x,y) ~d\alpha_2(y)\right) ~d\alpha_1(x). 
\end{equation}
In other words, Fubini's Theorem holds.  In particular, if $f$ is smooth, then \eqref{eqn:Fubini} holds.
\end{lemma}
\begin{proof}[Sketch of proof]
We have already observed above, for any $\pi_1(S)$-invariant H\"older continuous function $f: \mathcal G(\widetilde{S})\times \mathcal G(\widetilde{S})\to \RR$, that $\alpha_1 \otimes \alpha_2(f) = \alpha_1 \otimes \alpha_2(f|_{{\lambda_1}\pitchfork {\lambda_2}})$.  
Then $f|_{{\lambda_1}\pitchfork {\lambda_2}}$ has compact support contained in (a fundamental domain for) $\mathcal {DG}(S)$. 

For compactly supported $a$-H\"older continuous functions $f_i: \mathcal G(\widetilde S)\to \RR$, the product $f(x,y) = f_1(x)f_2(y)$ is compactly supported on $\mathcal G(\widetilde S)\times \mathcal G(\widetilde S)$ and $a$-H\"older continuous.    Equality \eqref{eqn:Fubini} clearly holds for such simple tensors, finite linear combinations of simple tensors, and limits of finite linear combinations of simple tensors, by linearity and continuity of the distributions $\alpha_1$, $\alpha_2$, and $\alpha_1\otimes \alpha_2$ restricted to $a$-H\"older functions.
Moreover, the closure of compactly supported smooth simple tensors $C_c^\infty(\mathcal G(Z))\otimes C_c^\infty(\mathcal G(Z))$ contains $C_c^\infty (\mathcal {DG}(Z))$  with respect to the $a$-H\"older norm for any $a\in (0,1)$, and this completes the proof of the lemma.  
\end{proof}

Now, given $\lambda_1, \lambda_2\in \mathcal {GL}_0(S)$ and $\alpha_i \in \cH(\lambda_i)$, we define the \emph{intersection number} 
\[ i(\alpha_1, \alpha_2) = \alpha_1\otimes \alpha_2(1) = \iint_{\mathcal {DG}(S)}~d\alpha_1 d\alpha_2 \in \RR.\]
The intersection number between transverse H\"older distributions is a natural extension of the intersection number of the space of measured laminations.  
\begin{lemma}\label{lem:intersection}
With respect to an auxiliary hyperbolic metric $Z\in \T(S)$, the intersection number can be computed on (geometric) train tracks $\tau_1=\tau(Z, \lambda_1, \epsilon_1)$ and $\tau_2=\tau(Z,\lambda_2, \epsilon_2)$ carrying $\lambda_1$ and $\lambda_2$ with small geodesic curvature, respectively as follows:
\begin{equation}\label{eqn:tt_integral}
i(\alpha_1, \alpha_2) = \sum_{p\in \tau_1\pitchfork \tau_2} \alpha_1(t_1(p)) \alpha_2(t_2(p)),
\end{equation}
where $t_i(p)$ is a tie of $\tau_i$ through $p$. 

The intersection number defines a bilinear pairing $\cH(\lambda_1) \times \cH(\lambda_2)\to \RR$, and if $\alpha_1$ and $\alpha_2$ represent tangent vectors to measures $\mu_1$ and $\mu_2$ with support contained in $\lambda_1$ and $\lambda_2$, respectively, then $i(\alpha_1, \alpha_2)$ is the second derivative of the geometric intersection pairing $i:\ML(S)\times \ML(S)\to \RR_{\ge0}$ at the point $(\mu_1, \mu_2)$ in the directions $\alpha_1$ and $\alpha_2$.  
\end{lemma}
\begin{proof}
That $\tau_1$ and $\tau_2$ are geometric with small geodesic curvature on $Z$ ensures that they are in minimal position.  The support of $\alpha_1 \otimes \alpha_2$ is contained in the transverse intersections between $\lambda_1$ and $\lambda_2$; take a geodesic quadrilateral $Q_p$ for each intersection $p\in \tau_1\pitchfork \tau_2$ with opposite sides contained in leaves of $\lambda_i$ and ties $t_1(p)$ and $t_2(p)$ running through $Q_p$.  Then $\iint_{Q_p} d\alpha_1 d\alpha_2 = \alpha_1(t_1(p)) \alpha_2(t_2(p))$ is immediate from Lemma \ref{lem:Fubini}.  By finite additivity and transverse invariance, any subdivision of $Q_p$ into smaller rectangles with smaller transversals yields the same result after summation.  

Note that \eqref{eqn:tt_integral} is identical to the expression for the geometric intersection pairing between measures carried by $\tau_1$ and $\tau_2$.  Indeed, the non-negative cone in the weight space of $W(\tau_i)$ defines a (possibly high codimension) subspace of measures in $\ML(S)$, for $i =1,2$, and the intersection pairing is bilinear restricted to pairs of measured laminations in the product of these convex cones; a coordinate expression is given by \eqref{eqn:tt_integral}.  Using the canonical isomorphisms $W(\tau_i) \cong \cH(\lambda_i)$, we have demonstrated bilinearity of $i: \cH(\lambda_1)\times \cH(\lambda_2) \to \RR$.

For each measure $\mu_i$, there is a convex cone with finitely many sides in $\cH(\lambda_i) \cong W(\tau_i)$ that represents a linear fragment of the one sided tangent space to $\mu_i$ in $\ML(S)$ (see \cite[Theorem 22]{Bon_GLTHB} or \S\ref{sub:THD}).  For each pair $\alpha_1$ and $\alpha_2$ in these tangent cones and small positive $s$ and $t$, $\mu_1 + s\alpha_1$ and $\mu_2+t\alpha_2$ define measured laminations carried by $\tau_1$ and $\tau_2$, respectively.  
As in the previous paragraph, the intersection form on non-positive cones extends by the same formula to $W(\tau_1)\times W(\tau_2)$.  To compute the second derivative, we can then write
\[\ddtp \left.\frac{d}{ds}\right|_{s=0^+} i(\mu_1 +s\alpha_1, \mu_2 +t\alpha_2)  =\ddtp(\alpha_1, \mu_2 +t\alpha_2) = i(\alpha_1, \alpha_2).\]
This completes the proof of the lemma.
\end{proof}

\subsection{Integration}
We will integrate transverse H\"older distributions with respect to H\"older continuous functions $f: k \to \RR$ on transverse arcs $k$.  The integral turns out to only depend on the restriction of $f$ to $k\cap \lambda$ \cite[Lemma 1]{Bon_THDGL}.  The fact that $k\cap \lambda$ has Hausdorff dimension $0$ \cite{BS} is an important ingredient in establishing the following integral formula.
\begin{theorem}[Theorem 10 of \cite{Bon_THDGL}]\label{thm:int_formula}
With notation as above and in Section \ref{subsec:geometric_tt}, the integral $\alpha(f) = \int_k f ~d\alpha$ is given by
\[\alpha(f) = \alpha(k)f(k^+) + \sum_{d\subset k\setminus \lambda} \alpha(k_d)\left(f(d^-) - f(d^+)\right),\]
where the sum is taken over all components $d\subset k\setminus \lambda$ except for the two extreme components.
\end{theorem}

Given a geodesic transversal $k$ to $\lambda$, and an $a$-H\"older function $f: k\cap \lambda \to \RR$, let 
\[\|f\|_a =\sup_x |f(x)| + \sup_{x\not= y} \frac{|f(x) - f(y)|}{|x-y|^a} \ \]
 denote the $a$-H\"older norm of $f$ restricted to $\lambda\cap k$.   We also define $ \alpha(k) = \int_k d\alpha $.  The following can be extracted from the proof of \cite[Theorem 11]{Bon_THDGL}, but we include a proof here for convenience of the reader.

\begin{lemma}\label{lem:angle_error}
Let $\lambda\in \mathcal {GL}_0(S)$,  $\alpha \in \cH(\lambda)$, $Z\in \T(S)$ be $\delta$-thick, $f\in \cH^a(k)$, and $\epsilon>0$.
Let $\tau_\epsilon = \tau(Z, \lambda, \epsilon)$ be a geometric train track, and let $k$ be a geodesic transversal to a branch $b$ of $\tau_\epsilon$.  
There exists $C>0$ depending on the topology of $S$, the angle that $k$ makes with $\lambda$, $\delta$, and $a$ such that if  $\epsilon$ is small enough, then
\[\left|\alpha(f) - \alpha(k)f(k^+)\right| \le C\|f\|_a\cdot \|\alpha\| \epsilon^a.\]
\end{lemma}
\begin{proof}
By Theorem \ref{thm:int_formula}, we have 
\begin{align*}
\left|\alpha(f) - \alpha(k)f(k^+)\right| & \le \sum_{d} \left| \alpha(k_d)(f(d^-) - f(d^+)    \right| \\
& \le  \sum_{d} r(d) \|\alpha\| \|f\|_a d(d^-,d^+)^a \\
&\le 6|\chi(S)| \cdot \|\alpha\| \cdot \|f\|_a\sum_{r=1}^\infty r A^a  \epsilon^a \exp\left( \frac{-ar 2\delta}{9|\chi(S)|}\right )\\
& \le  C \|\alpha\| \cdot \|f\|_a \epsilon^a.
\end{align*}

We have used Proposition \ref{prop:tt_geometry} item (2) and all three items of Lemma \ref{lem:depth} to justify the inequalities.  
\end{proof}

In general, we can approximate $\alpha_1\otimes \alpha_2(f)$ by ``Riemann sums'' of the form 
\begin{equation}\label{eqn:R_sum_prod}
(\alpha_1\otimes \alpha_2)_\epsilon(f) := \sum_{p\in \tau_1(\epsilon)\pitchfork \tau_2(\epsilon)} f(p) \alpha_1(t_1(p))\alpha_2(t_2(p)).
\end{equation}
where $t_i(p)$ is a tie for $\tau_i(\epsilon) = \tau(\lambda_i, Z, \epsilon)$  through points $p$ chosen in the set of transverse intersections $\lambda_1\pitchfork \lambda_2$ corresponding to  the intersections $\tau_1(\epsilon)\pitchfork \tau_2(\epsilon)\subset Z$, for $i = 1,2$. 

Setting notation for the proposition, take $\epsilon_0$ small enough so that $\tau_i(\epsilon_0)$ both define train track neighborhoods, and $W(\tau_i(\epsilon_0))\cong \cH(\lambda_i)$.  By $\|\alpha_i\|$, we mean the $\ell^\infty$ norm on $\RR^{b(\tau_i(\epsilon_0))}$. 

\begin{proposition}\label{prop:R_sums_prod}
With notation as above, then  the Riemann sums $(\alpha_1\otimes \alpha_2)_\epsilon(f)$ converge to $\alpha_1\otimes \alpha_2(f)$ as $\epsilon\to 0$, for any $\pi_1(S)$-invariant $a$-H\"older function $f: \mathcal {G}(\widetilde S) \times  \mathcal {G}(\widetilde S) \to \RR$ in the closure of the linear span of simple tensors $\cH^a ( \mathcal {G}(\widetilde S)) \otimes \cH^a(  \mathcal {G}(\widetilde S))$.  
Moreover the convergence is at rate
\[ | (\alpha_1 \otimes \alpha_2)_\epsilon(f) - \alpha_1 \otimes \alpha_2(f)|  = O\left( \epsilon^{b}\right), \]
where $b\in (0,a)$. The implicit constants depend (continuously)  on $b$, $\|f\|_a$, $\|\alpha_1\|$, $\|\alpha_2\|$, $\inj(Z)$, and the topology of $S$. 
\end{proposition}

\begin{proof}
First we note that $\alpha_1\otimes \alpha_2$ is supported on $\lambda_1\pitchfork \lambda_2$, and on our reference surface $Z$, these transverse intersections are contained in the union $\cup Q_p$ of geodesic quadrilaterals, as in the proof of Lemma \ref{lem:intersection}. So, we may think of $f$ as a H\"older continuous function on $Z$, supported on $\lambda_1\pitchfork \lambda_2$, and write
\begin{align}\label{eqn:sum_p}
\left| \alpha_1\otimes \alpha_2 (f) - (\alpha_1\otimes \alpha_2)_\epsilon (f) \right| &= \left| \sum_{p} \iint_{Q_p} f - f(p) ~d\alpha_1d\alpha_2 \right|.
\end{align}

We now take a moment to point out that the usual integral inequality $| \int f~d\mu| \le \int |f| ~d\mu$ for measures $\mu$ does not necessarily hold for general distributions, as our `integrals' are signed.  As such, we will need to use care when distributing absolute value signs.  For instance, by the triangle inequality and Fubini's Theorem for distributions which are H\"older approximated by simple tensors (Lemma \ref{lem:Fubini}), we can bound 
\begin{align*}
\eqref{eqn:sum_p} \le \sum_p \left| \int \left( \int f(x, y) -f(x_p, y_p) ~d\alpha_1(x)\right) ~d\alpha_2 (y)\right| ,
\end{align*}
where $x$ represents a geodesic segment of $\lambda_1\cap Q_p$ and $y$ represents a geodesic segment of $\lambda_2\cap Q_p$, so that $(x,y)\in Z$ represents their intersection, and we take $p = (x_p, y_p)$.  

For $y$ fixed, we denote $f_y(x) = f(x,y)$ and define $F_p(y) = \int f_y(x) - f(x_p, y_p)~d\alpha_1(x)$.  In order to evaluate the iterated integral above, we need to know that $F_p(y)$ is $b$-H\"older continuous for some $b$. 

\begin{claim}\label{clm:F_Holder}
For $\epsilon$ small enough and all $p$, $F_p(y)$ is $b$-H\"older continuous for $b\in (0,a)$.
\end{claim}
\begin{proof}[Proof of Claim \ref{clm:F_Holder}]
Let $y_1, y_2$ be segments of leaves of $\lambda_2\cap Q_p$ at Hausdorff distance $\Delta <E\epsilon$, which is a bound for the width of the $\epsilon$-horocyclic neighborhood of $\lambda_2$ (Proposition \ref{prop:tt_geometry} item (3)).  
We consider 
\begin{align}\label{eqn:F_Holder_1}
|F_p(y_1) - F_p(y_2)| \le  | \alpha_1 (f_{y_1} - f_{y_2})  - \alpha_1(t_1(p))&(f_{y_1}(x^+) -f_{y_2}(x^+)) | \\ 
& + |\alpha_1(t_1(p))| \cdot |(f_{y_1}(x^+) -f_{y_2}(x^+))|,\nonumber
\end{align}
where $x^+$ is the last intersection of $y_1$ and $y_2$ with a leaf of $\lambda_1$ in $Q_p$.  Note that the intersections $y_1$ with $\lambda_1$ and $y_2$ with $\lambda_2$ are in bi-H\"older correspondence, so it makes sense to compare these functions.     Using the fact that $f$ is $a$-H\"older and Lemma \ref{lem:norm}, we know 
\begin{align} \label{eqn:F_Holder_easy}
|\alpha_1(t_1(p))| \cdot |(f_{y_1}(x^+) -f_{y_2}(x^+))| \le C\|\alpha\| \log(1/\epsilon) \|f\|_a \Delta^a.
\end{align}

Now we focus on bounding the first term of \eqref{eqn:F_Holder_1}, which by Theorem \ref{thm:int_formula} is 
\begin{align}\label{eqn:F_Holder_sum}
\left| \sum_{d\subset k \setminus \lambda_1} \alpha(k_d) \left(f_{y_1}(d^-) -f_{y_2}(d^-)- f_{y_1}(d^+) + f_{y_2}(d^+)\right)  \right| ,
\end{align}
where the sum is taken over components $d$ not containing an endpoint of a geodesic transversal $k$ (which we could take to be $y_1$ or $y_2$, for example).  

We separate this sum into two pieces.  Namely, we organize the components $d$ according to their divergence radius (measured with respect to $\tau_1(\epsilon)$), and separate out those which have depth at most $k$, whose value will be determined later.  
Proceeding as in the proof of Lemma \ref{lem:angle_error}, we can bound the previous equation by
\begin{align*}
  \sum_{r(d)\le k} |\alpha( k_d)||(f_{y_1}(d^-) &- f_{y_2}(d^-))  +(f_{y_2}(d^+) - f_{y_1}(d^+)) |  \\ 
  & + \sum_{r(d)>k} |\alpha( k_d)||(f_{y_1}(d^-) - f_{y_1}(d^+)) +(f_{y_2}(d^+) - f_{y_2}(d^-)) | \\
 & \le 6|\chi(S)| C \log(1/\epsilon)\|\alpha\|\cdot \|f\|_a \left( 2k\Delta^a + \sum_{r=k}^\infty 2r A^a  \epsilon^a \exp\left( \frac{-ar 2\delta}{9|\chi(S)|}\right ) \right) \\
  & \le 6|\chi(S)| C \log(1/\epsilon)\|\alpha\|\cdot \|f\|_a \left( 2k \Delta^a + 2C'k\epsilon^a\exp\left( \frac{-a2\delta(k-1) }{9|\chi(S)|}\right )\right) 
\end{align*}
  We want $\Delta^a \approx k\epsilon^a\exp\left( \frac{-a2\delta(k-1) }{9|\chi(S)|}\right)$, which means we should choose $k$ on the order of $\frac{9|\chi(S)|}{2\delta} \log(\epsilon/\delta)$.

Substituting $k$ chosen above back in, we obtain the upper bound
$\eqref{eqn:F_Holder_sum} = O(\Delta^b), $
where $b\in (0,a)$.  Together with \eqref{eqn:F_Holder_easy}, we have bounded \eqref{eqn:F_Holder_1} by $O(\Delta^b)$ for $b\in (0,a)$.  Since $\Delta$ was the distance between $y_1$ and $y_2$, this completes the proof of the claim that $F_p$ is $b$-H\"older continuous.  
\end{proof}

We have shown that the iterated integrals, summed over $p$, make sense, and therefore proceed to bound them.  We now recall that Proposition \ref{prop:tt_geometry} item (5) tells us that there are at most $O(\log^21/\epsilon)$ intersection points $p$, and so it suffices to bound any of the integrals in the sum.  
As such we want to bound 
\begin{align}
\left| \int F_p(y)~d\alpha_2(y)\right| & \le \left| \int F_p(y) - F_p(y^+)~d\alpha_2(y) \right| +|\alpha_2(t_2(p))|\cdot |F_p(y^+)| \nonumber\\
& \le C\|F_p\|_b \|\alpha_2\| \epsilon^{b} + C\log(1/\epsilon) \|\alpha_2\| |F_p(y^+)|, \label{eqn:almost_final_bound}
\end{align}
Where we have used Lemma \ref{lem:angle_error} and Lemma \ref{lem:norm} to move from the first to the second line of the inequality.  If now we can provide a good bound for $\|F_p\|_b$ in terms of $\epsilon$, then we will arrive at the proposition.
\begin{claim}\label{clm:F_bound_b}
We have bounds
\[|F_p(y)|  = O\left(\epsilon^{b}\right) \]
and $\|F_p\|_b = O\left(\epsilon^{b}\right) +O(1)$, where $b\in (0,a)$.  
\end{claim}
\begin{proof}[Proof of Claim \ref{clm:F_bound_b}]
Again we have 
\begin{align*}
|F_p(y)|  & =  \left |\int f_y(x) - f(x_p, y_p) ~d\alpha_1(x) \right| \\
& \le \left| \int f_y(x) - f_y(x^+) ~d\alpha_1(x) \right| + |\alpha_1(t_1(p))|\cdot |f(x^+,y) - f(x_p, y_p)|\\
& \le C \|f_y\|_a \|\alpha_1\| \log(1/\epsilon) \epsilon^{a} + C \|\alpha_1\|\log(1/\epsilon)  \|f\|_a (2E\epsilon)^a = O\left( \epsilon^{b}  \right),
\end{align*}
for $b\in (0,a)$.  Going from the second line to the third, we invoke Lemma \ref{lem:angle_error}, Lemma \ref{lem:norm}, and the Proposition \ref{prop:tt_geometry} item (3) to see that $\diam (Q_p) \le 2E\epsilon$; it is also clear that $\|f_y\|_a\le \|f\|_a$.  

Thus we have bounded $\|F_p\|_\infty$, while the proof of Claim \ref{clm:F_Holder} provides the second bound of $O(1)$, independent of $\epsilon$, for $\sup_{y_1\not=y_2}\frac{|F_p(y_1) - F_p(y_2)|}{d(y_1, y_2)^b}$.
\end{proof}
Putting Claim \ref{clm:F_bound_b} together with \eqref{eqn:almost_final_bound}, we have 
\[\left| \int F_p(y)~d\alpha_2(y) \right|  = 
O\left(\epsilon^{b}\right),\]
with $b\in (0,a)$. 
We have already observed that the quantity \eqref{eqn:sum_p} which we are trying to bound is the sum over  $O(\log^2(1/\epsilon))$-many points $p$ of terms given above.  This gives the final estimate $O(\log^2(1/\epsilon) \epsilon^{b}) = O(\epsilon^{c})$ for any $c\in (0,a)$, and concludes the proof of the proposition.
\end{proof}

\subsection{Shearing deformations}
We recall the construction of shearing deformations and shearing cocycles introduced in \cite[\S5]{Bon_SPB}, although we follow (more) closely the exposition  and notation given in   \cite{BonSoz:variation}.
We fix a hyperbolic structure $Z\in \T(S)$ and maximal chain recurrent geodesic lamination $\lambda\in \mathcal {GL}_0(Z)$ and consider a transverse H\"older cocycle $\alpha\in \cH(\lambda)$.
For small time $t\le t_Z$ depending on the geometry of $\lambda\subset Z$ and the size of $\alpha$, there is a $\PSL_2\RR$-valued \emph{shearing cocycle} 
\[E^{t\alpha}: \left(\tX\setminus\tlambda\right)^2\to \Isom^+(\tX)\cong \PSL_2\RR,\]
 which is constant on the triangular complementary components.   
 
 Intuitively, $E^{t\alpha}$ explains how to cut $\tX$ apart along $\tlambda$ and reglue equivariantly by right- and left-earthquakes.  We will see that the recipe for this deformation involves infinite ordered products of parabolic isometries of the plane, adjusting the shearing of complementary triangles relative to one another, and convergence is not automatic (as would essentially be the case if we were dealing with right-earthquakes or left-earthquakes, exclusively).   
 Formally, the shearing cocycle defines a one-parameter analytic family of discrete and faithful representations $\rho_{t\alpha}: \pi_1(Z,z) \to \Isom^+\tX$, given by post-composition: if $\gamma\in \pi_1(Z,z) \le \Isom^+\tX$, then 
 \[\rho_{t\alpha}(\gamma) := E^{t\alpha}(\Tx, \gamma \Tx)\circ \gamma,\]
  where $T_z\subset \tX\setminus \tlambda$ is a fixed lift of a triangle containing a lift of the basepoint $z$ \cite{Bon_SPB}.  We hence obtain an analytic one-parameter family $E^{t\alpha}Z = [\tX/\rho_{t\alpha}]\in \T(S)$.   
  
 With $\Tx$ fixed, we can define a partial order $<$ on the components  of $\tX\setminus\tlambda$ different from $\Tx$: for such components $T$ and $T'$, we declare that $T<T'$ if  $T$ separates $T'$ from $\Tx$.  A  non-backtracking path $k_T$ starting at $\Tx$ and ending in some component $T$ gives an orientation to all geodesics of $\tlambda$ that it meets making the intersection positive with respect to the underlying orientation of $\HH^2$.
 Each component $T'<T$ has two oriented boundary geodesics $g_2^{T'}$ and $g_1^{T'}$ meeting $k_T$, where  $g_2^{T'}$ is closest to $\Tx$.  Let $k_{T'}$  be the restriction of $k_T$ from it's start to some point of intersection with $T'$, and take $\alpha(T') = \alpha(k_{T'}) = \alpha(\Tx, T')$.
 
 For an oriented geodesic $g\subset\HH^2$, $E_g^a$ is the hyperbolic translation along $g$ with (signed) magnitude $a\in \RR$.   
  Define $E_T^a:=E_{g_2^T}^a \circ E_{g_1^T}^{-a}$ for $a\in \RR$, called the \emph{elementary shear in $T$}.
  
  We may now define 
  \begin{equation*}\label{eqn:shear_cocycle}
  E^{\alpha}(\Tx, T) : = \overrightarrow {\prod_{T'\le T}} E_{T'}^{\alpha(T')} \prod E_{g_2^T}^{\alpha(T)},
  \end{equation*}
  where the first (ordered) infinite product is the limit
    \begin{equation}\label{eqn:approx}
    \lim E_{T_1}^{\alpha(T_1)}\circ E_{T_2}^{\alpha(T_2)} \circ ... \circ E_{T_n}^{\alpha(T_n)}, 
      \end{equation}
    as the chain $T_1< T_2< ... <T_n<T$ increases to the set of all triangles smaller than $T$. 
    It was shown in \cite[\S5]{Bon_SPB} that the limit exists for $\alpha$ small enough compared to $\inj(Z)$, and defines a $\pi_1(Z,z)$-equivariant cocycle.  
 
 Under the shear coordinates described earlier (Theorem \ref{thm:coordinates}), we have 
 \[\sigma_\lambda(E^{t\alpha}Z) = \sigma_\lambda(Z) + t\alpha\in \cH^+(\lambda).\]  
 
\section{A cosine formula}\label{sec:cosine}
The goal of this section is to compute the Poisson bracket between length functions for transverse H\"older distributions on transversally intersecting geodesic laminations.
We consider the change of length of a transverse cocycle supported on $\lambda_1$ under a deformation in the direction of a transverse cocycle supported on $\lambda_2$.
Without loss of generality, we may assume that $\lambda_1$ and $\lambda_2$ are both maximal.

More precisely, let $\alpha_1\in \cH(\lambda_1)$ and $\alpha_2\in \cH(\lambda_2)$, and consider the \emph{shearing vector fields} $X_{\alpha_1} : =(\sigma_{\lambda_1}\inverse)_*\alpha_1$ and $X_{\alpha_2} : = (\sigma_{\lambda_2}\inverse)_*\alpha_2$ on $\T(S)$.
We would like to compute $d\ell_{\alpha_1} X_{\alpha_2}$, which by Lemma \ref{lem:length_dual}, is equivalent to computing the Poisson bracket 
\[\{ \ell_{\alpha_1}, \ell_{\alpha_2}\} = \omega_{\WP} (X_{\alpha_1},X_{\alpha_2}) = -\{ \ell_{\alpha_2}, \ell_{\alpha_1}\}.\]

Note that when $\mu_1 \in \cH(\lambda_1)$ and $\mu_2 \in \cH(\lambda_2)$ are (positive) measures, then Kerckhoff's  cosine formula \cite[Lemma 3.2]{Kerckhoff:NR} (see also \cite[Theorem 3.3]{Wolpert:symplectic}) gives 
\begin{equation}\label{eqn:EQ}
\{\ell_{\mu}, \ell_{\nu}\} (Z) = \iint_Z \cos (\lambda_1, \lambda_2)~ d\mu_1 d\mu_2,
\end{equation}
where if $p\in \lambda_1\pitchfork \lambda_2$, then $\cos(\lambda_1, \lambda_2)(p)$ is the cosine of the angle from the corresponding leaf of $\lambda_1$ measured in the positive direction to the corresponding leaf of $\lambda_2$ near $p$ on $Z$.  

Let $\cQ$ be a component of a stratum of quadratic differentials with associated $\cQ$-Thurston measure $\mu_{\Th}^\cQ$ (c.f. \S\ref{sub:Thurston}).
Our main result in this section is the following generalization of \eqref{eqn:EQ}.
\begin{theorem}\label{thm:cosine}
For $\mu_{\Th}^\cQ$ almost every $\mu \in \ML(S)$ with chain recurrent diagonal extension $\lambda_1$, for every chain recurrent geodesic lamination $\lambda_2$,  $\alpha_2\in \cH(\lambda_2)$,  and $Z\in \T(S)$,  if $\alpha_1\in \cH(\lambda_1)$ represents a tangent vector to $\mu$ in $\ML(S)$, then   \[\{\ell_{\alpha_1}, \ell_{\alpha_2}\} (Z) = d\ell_{\alpha_1}X_{\alpha_2}(Z) = \iint_Z \cos(\lambda_1, \lambda_2) ~d\alpha_1d\alpha_2.\]
\end{theorem}
\begin{remark}
We only use the restrictive hypotheses on $\mu$ and and $\lambda_1$ in the final limiting argument, so Lemma \ref{lem:curve} and Proposition \ref{prop:cosine} hold without qualification for all $\alpha_1\in \cH(\lambda_1)$ and $\alpha_2\in\cH(\lambda_2)$, and $\lambda_1$ is allowed to be any chain recurrent geodesic lamination. It is natural to ask whether this restriction is necessary or not, especially given the asymmetry of theorem and the (anti-)symmetry of the Poisson bracket.
\end{remark}

We may sometimes denote the corresponding function as $\Cos(\alpha, \beta) : \T(S)\to \RR$.  
Our proof of Theorem \ref{thm:cosine}, which occupies the rest of this section, follows the strategy of \cite[\S II.3.9]{EM} for simple curves and the limiting argument of \cite{Kerckhoff:analytic} for measures on  laminations.  A more delicate limiting argument proves the theorem for general transverse H\"older distributions.  
\begin{lemma}\label{lem:curve}
Let $\gamma$ be a closed geodesic in $Z$, $\lambda$ be a chain recurrent geodesic lamination and $\alpha\in \cH(\lambda)$.  Then 
\[\{\ell_\gamma, \ell_\alpha\}(Z) = d\ell_\gamma X_\alpha (Z) = \int_\gamma \cos(\gamma, \lambda)~d\alpha = \iint_Z \cos(\gamma, \lambda)~d\alpha d\gamma.\]
\end{lemma}
\begin{proof}
This is a relatively straightforward computation following \cite[\S II.3.9]{EM}; we include the details for completeness.  

Let $\gamma\in \pi_1(Z,z) \le \Isom^+(\widetilde Z) \cong \PSL_2\RR$.  The absolute value $\tau(\gamma) = \tau$ of the trace of either lift of $\gamma$ to $\SL_2\RR$ is related to the translation length $\ell(\gamma) = \ell$ by 
\[ 2\exp (\ell/2) = \tau + \sqrt{\tau^2-4}. \]
If $\gamma_t$ is an analytic path in $\PSL_2\RR$ with $\gamma_0 = \gamma$, then 
\[\frac{d\ell}{dt}= \frac{2}{\sqrt{\tau^2 -4}}\frac{d\tau}{dt}.\]

Complete $\lambda$ to a maximal chain recurrent geodesic lamination, and view $\alpha$ as a transverse cocycle on the new object.  For $T< \gamma. T_x$, and $t\ge0$, we would like to calculate the trace $E_T^{t\alpha(T)}\gamma$, which is invariant by conjugation.  We may therefore normalize the situation in the upper half plane as follows:
$g_2^{T}$ is the geodesic between $0$ and $\infty$, $g_1^T$ is the geodesic between $1$ and infinity, and $\gamma$ goes from $a<0$ to $b>0$.  The matrix representing $\gamma$ in $\SL_2\RR$ with respect to this normalization is given by 
\[ \frac{1}{b-a} \begin{pmatrix}
b e^{\ell/2} - ae^{-\ell/2} & ab(e^{-\ell/2} - e^{\ell/2}) \\
e^{\ell/2} - e^{-\ell/2} &  b e^{-\ell/2}-a e^{\ell/2}
\end{pmatrix},
\]
While
\[E_{T}^{t\alpha}  = E_{g_2^T}^{t\alpha(T)} \circ E_{g_1^T}^{-t\alpha(T)} = 
\begin{pmatrix}
e^{\alpha(T)/2} & 0 \\
0 & e^{-\alpha(T)/2}
\end{pmatrix}
\begin{pmatrix}
1 & 1\\
0 & 1
\end{pmatrix}
\begin{pmatrix}
e^{-\alpha(T)/2} & 0 \\
0 & e^{\alpha(T)/2}
\end{pmatrix}
\begin{pmatrix}
1 & -1\\
0 & 1
\end{pmatrix}.
\]

Direct computation along this path yields
\[\left.\frac{d\tau }{dt}\right|_{t=0} = - \alpha(T) \frac{\sinh\ell/2}{a-b},\]
so that 
\[ \left.\frac{d\ell}{dt}\right|_{t=0}   = -2\alpha(T) \frac{ \sinh\ell/2}{a-b} \frac{1}{\sqrt{\cosh^2\ell/2 - 4}} = \alpha(T) \frac{-2}{a-b}.\]
The center of the Euclidean circle joining $a$ to $b$ is $(b+a)/2$ and its radius is $(b-a)/2$.  An exercise in Euclidean trigonometry shows that 
\[ -\cos(\gamma, g_2^T) = \frac{b+a}{2}/\frac{b-a}{2} = \frac{b+a}{b-a} \text{  and } \cos(\gamma, g_1^T) = \left(1 - \frac{b+a}{2}\right) /\frac{b-a}{2} = \frac{2-(b+a)}{b-a}, \]
where $\cos(\gamma, g_i^T)$ is the cosine of the angle from $\gamma$ to $g_i^T$ in the positive direction.
Thus 
\[\left.\frac{d\ell}{dt}\right|_{t=0}  =\alpha(T) \left(\cos(\gamma, g_2^T) -\cos(\gamma, g_1^T) \right). \]
Similarly, the derivative of length at time $t=0$ for $E_{g_2^{\gamma.{\Tx}}}^{t\alpha(\gamma.{\Tx})}\gamma$ is given by 
\[ \left.\frac{d\ell}{dt}\right|_{t=0} = \alpha(\gamma.\Tx) \cos(\gamma, g_2^{\gamma.{\Tx}}).\]

For an approximating chain $T_1 < ... < T_n<\gamma.\Tx$,  define 
\[\rho_{t,n} (\gamma) : = E_{T_1}^{t\alpha(T_1)} \circ ... \circ E_{T_n}^{t\alpha(T_n)}\circ E_{g_2^{\gamma.{\Tx}}}^{t\alpha(\gamma.{\Tx})}\circ \gamma \in \PSL_2\RR, \]
so that 
\[\lim_{n\to \infty} \rho_{t,n}(\gamma) =\rho_t(\gamma) \]
uniformly on compact sets in $\HH^2$.
Then
\[ \left.\frac{d}{dt}\right|_{t =0} \ell( \rho_{t,n}(\gamma)) = \sum_{i =1}^n \alpha(T_i) \left(\cos(\gamma, g_2^{T_i}) -\cos(\gamma, g_1^{T_i}) \right) +\cos(\gamma, g_2^{\gamma.\Tx}). \]

The function $\cos(\gamma, \mu)$ is $c$-Lipschitz for some universal $c>1$ \cite[Lemma 1.1]{Kerckhoff:NR}, 
so by Theorem \ref{thm:int_formula}, this series converges to $\alpha\left ( \cos ( \gamma, \lambda)\right)$ as the chain $T_1< ... <T_n<\gamma.\Tx$ increases to  $\{T : T< \gamma.Tx\}$.  

The family $\rho_t(\gamma) =  E^{t\alpha}(\Tx,\gamma.\Tx)\circ \gamma$ varies holomorphically in a small complex parameter $t$ \cite[Theorem 31]{Bon_SPB}.  By holomorphicity, we may exchange the limit: 
\[\left.\frac{d}{dt}\right|_{t=0} \ell(\rho_t(\gamma)) = \left.\frac{d}{dt}\right|_{t =0} \lim_{n\to \infty}\ell(\rho_{t,n}(\gamma)) = \lim_{n\to \infty} \left.\frac{d}{dt}\right|_{t=0} \rho_{t,n}(\gamma) =\alpha\left(\cos(\gamma, \lambda)\right),\]
which is what we wanted to prove.
\end{proof}
Given maximal $\lambda_1$ and $\lambda_2 \in \mathcal {GL}_0(S)$, let $\tau_1(\epsilon)$ and $\tau_2(\epsilon)$ be geometric train tracks constructed from $(Z,\lambda_1, \epsilon)$ and $(Z,\lambda_2,\epsilon)$, respectively, where $\epsilon$ is small enough as in \S\ref{subsec:geometric_tt}.  Let $p$ denote a point of transverse intersection, and $Q_p$ be the corresponding geodesic quadrilateral with opposite sides contained in leaves of $\lambda_1$ and $\lambda_2$.  
We establish the following special case of Theorem \ref{thm:cosine} by a limiting argument; compare with \cite[Proposition 2.5]{Kerckhoff:analytic}. Most of the technical work for the proof was done in \S\ref{sec:int}. 

\begin{proposition}\label{prop:cosine}
For any  $\lambda_1$ and $\lambda_2$, chain recurrent geodesic laminations on $Z \in \T(S)$, $\alpha\in \cH(\lambda_1)$, and $\nu\in \ML(S)$, a measure supported in $\lambda_2$, we have
\[\{\ell_\nu, \ell_\alpha\}(Z) = \iint_Z\cos(\lambda_2,\lambda_1)~d\alpha d\nu.\]
\end{proposition}

\begin{proof}
It suffices to prove that if $\gamma_i$ is a sequence of weighted multicurves converging in measure to $\nu$, then 
\[\iint_Z \cos(\gamma_i,\lambda_1) ~d\alpha d\gamma_i \to \iint_Z \cos(\lambda_2,\lambda_1)~d\alpha d\nu,\]
uniformly on compact subsets of $\T(S)$.  Indeed, we know that $\ell_{\gamma_i} \to \ell_{\nu}$ and that $d\ell_{\gamma_i}X_\alpha = \iint \cos(\gamma_i, \lambda_1) ~d\alpha d\gamma_i$ by Lemma \ref{lem:curve}, so if the convergence is uniform near $Z$, then $d\ell_{\nu}X_\alpha(Z) = \iint_Z \cos(\lambda_2,\lambda_1)~d\alpha d\nu$.

We consider  $\epsilon_0$ small enough (as in Proposition \ref{prop:tt_geometry} item (4)) and construct $\tau_1(\epsilon_0)$ and $\tau_2(\epsilon_0)$, endowing the weight spaces with the restriction of $\ell^\infty$ norms $\| \cdot \|$ on $\RR^{b(\tau(\epsilon_0))}$, $i = 1,2$.  
For $\epsilon<\epsilon_0$, we consider the geodesic quadrilaterals $Q_p$ corresponding to intersection points $p \in \mu\pitchfork \nu$ near  transverse intersection points $\tau_1(\epsilon)\pitchfork \tau_2(\epsilon)$.  
The functions $\cos_Z(\lambda_2,\lambda_1)$ and $\cos_Z(\gamma_i, \lambda_1)$ are restrictions of a smooth function $\cos_Z: \mathcal {DG}(Z)\to \RR$ to the support of $\alpha\otimes \nu$ and $\alpha\otimes \gamma_i$, respectively; see \S\ref{sec:int}. 

In particular, the derivative of $\cos_Z$ is bounded on the (compact) support of $\cos_Z(\lambda_2,\lambda_1)$; thus $\cos_Z(\lambda_2,\lambda_1)$ is Lipschitz (in fact $c$-Lipschitz for a universal constant $c$ \cite[Lemma 1.1]{Kerckhoff:NR}). 
We can therefore apply Proposition \ref{prop:R_sums_prod} to see that 
\[\left| (\alpha\otimes \nu)_\epsilon(\cos_Z(\lambda_2,\lambda_1)) - \alpha\otimes \nu(\cos_Z(\lambda_2,\lambda_1))\right| = O(\epsilon^{a}),\]
for $a\in (0,1)$ and where the implicit constants depend continuously on $\inj(Z)$,  the Lipschitz constant $c$ of $\cos_Z(\lambda_2,\lambda_1)$  as well as $\|\alpha\|$, $\|\nu\|$, $a$, and the topology of $S$.    

For a fixed $\epsilon$ and  $i$ large enough, $\gamma_i$ is carried fully by $\tau_2(\epsilon)$, i.e. gives every branch of $\tau_2(\epsilon)$ positive measure.  We can agin apply the proof of Proposition \ref{prop:R_sums_prod} to see that 
\[\left| (\alpha\otimes \gamma_i)_\epsilon(\cos_Z(\gamma_i,\lambda_1)) - \alpha\otimes \nu(\cos_Z(\gamma_i,\lambda_1))\right| = O(\epsilon^{a}),\]
for $a\in (0,1)$ and implicit constants depending all all of the same data replacing $\|\nu\|$ with $\|\gamma_i\|$.  

It therefore suffices to compare the Riemann sums directly; we borrow notation for the Riemann sum from the proof of Proposition \ref{prop:R_sums_prod}, and choose the points $p_i \in \gamma_i \cap \lambda_1$ and $p\in \lambda_2\cap \lambda_1$ on the same leaf of $\lambda$.
\begin{align*}
\left| (\alpha\otimes \gamma_i)_\epsilon(\cos_Z(\gamma_i,\lambda_1)) - (\alpha\otimes \nu)_\epsilon(\cos_Z(\lambda_2,\lambda_1)) \right| \le \sum_{\tau_1(\epsilon) \pitchfork \tau_2(\epsilon)} &| \cos_Z(\gamma_i,\lambda_1)(p_i)\gamma_i(t_2(p)) \\ &- \cos_Z(\lambda_2,\lambda_1)(p)\nu(t_2(p))| |\alpha(t_1(p))|
\end{align*}
holds for $i$ large enough compared to $\epsilon$.  Furthermore, passing to a subsequence, we can assume that $|\gamma_i(t_2(p)) - \nu(t_2(p))| = O(\epsilon)$ for all $p$ and $i$ large enough.  

Since $\gamma_i$ is carried on $\tau_2(\epsilon)$, it is contained in an $O(\epsilon)$ neighborhood of $\lambda_2$.  
We can apply Lemma \ref{lem:triple_angle_close} to obtain 
\[|\cos_Z(\gamma_i, \lambda_1)(p_i) - \cos_Z(\lambda_2,\lambda_1)(p)| = O\left(\epsilon\right),\] for $i$ large enough.  As in the proof of Proposition \ref{prop:R_sums_prod}, there are at most $O(\log^21/\epsilon)$ intersection points over which we sum, and $|\alpha(t_1(p))|  = O(\log1/\epsilon)$ (Lemma \ref{lem:norm}).  Combining these observations, we can bound the difference of the Riemann sums by 
\[ O\left(\log^3(1/\epsilon)\cdot \epsilon\right) = O(\epsilon^{a}),  \]
for $a\in (0,1)$, and all implicit constants depending continuously $\inj(Z)$ and the other relevant data.  

By the triangle inequality, convergence of the Riemann sums approximating $\iint_Z\cos(\lambda_2,\lambda_1)~d\alpha d\nu$ are approximated uniformly by those approximating $\iint_Z\cos(\gamma_i, \lambda_1)~d\alpha d\gamma_i$ on compact sets of $\T(S)$, and the proposition follows.
\end{proof}

To prove the general statement of Theorem \ref{thm:cosine}, without loss of generality, we regard $\alpha_1\in \cH(\lambda_1)$ as tangent vector to the measure $\mu$ and consider a family of difference quotients
\[\ell_s = \frac{\ell_{\mu + s\alpha_1} -\ell_{\mu}}{s},\]
where $\mu_s := \mu+s\alpha_1$ represents a measured lamination carried on a train track $\tau_1(\epsilon_0) = \tau(Z, \lambda_1, \epsilon_0)$ for small enough $s$.   

We know that $\ell_s$ are analytic functions converging to the analytic function $\ell_{\alpha_1}$ on $\T(S)$ (see \eqref{eqn:dlength}).  
We want to prove that the derivatives converge to the cosine formula from the theorem statement uniformly near $Z$ as $s\to 0$.
\begin{proof}[Proof of Theorem \ref{thm:cosine}]
By Proposition \ref{prop:cosine}, we have 
\begin{equation}\label{eqn:ell_s}
d\ell_s X_{\alpha_2} = \frac1s\left( \iint \cos(\mu_s,\lambda_2)~d\alpha_2 d\mu_s - \iint \cos(\lambda_1, \lambda_2) ~d\alpha_2 d\mu \right) 
\end{equation}
 
 We have assumed that $\mu$ is generic in the sense of Theorem \ref{thm:measure_H_close}, so there is a sequence $\epsilon_1, \epsilon_2, ...$ tending to zero, constants $c_1$ and  $c_2$ and $a\in(0,1]$ depending continuously on $Z$ such that if we take $s_n = c_1(\log1/\epsilon_n)^{-2}$, then a $c_2\epsilon_n^a$-neighborhood of $\lambda_1$ contains $\mu_{s_n}$.   
 Let $\epsilon_n' = c_2\epsilon_n^a$.   

As in the proof of Proposition \ref{prop:cosine}, each of the two terms in \eqref{eqn:ell_s} is uniformly well approximated by their Riemann sums at scale $\epsilon_n'$.  Since $\mu_{s_n}\prec \tau_1(\epsilon_n')$, we know that $\mu_{s_n}(t_b) = \mu(t_b) + s_n \alpha_1(t_b)$ for all branches $b$ or $\tau_1(\epsilon_n')$.  Borrowing notation from Proposition \ref{prop:R_sums_prod}, we can choose $p_n \in \mu_{s_n}\cap \lambda_2$ and $p\in \lambda_1 \cap \lambda_2$ on the same leaf of $\lambda_2$ as in the proof of Proposition \ref{prop:cosine}.  These points at distance at most $O(\epsilon)$ apart; we abuse notation and write them as the same point.  We can approximate the first term in \eqref{eqn:ell_s} 
by
\begin{align*}
\sum_{p\in\tau_1(\epsilon_n') \pitchfork \tau_2(\epsilon_n')} &\frac{1}{s_n} \cos(\mu_{s_n}, \lambda_2)(p)\cdot  \alpha_2(t_2(p)) \mu_{s_n}(t_1(p)) 
\\&= \sum_{p} \frac{1}{s_n} \cos(\mu_{s_n}, \lambda_2)(p)\cdot  \alpha_2(t_2(p)) (\mu(t_1(p)) + s_n\alpha_1(t_1(p)))\\
&=\sum_{p} \frac{ \cos(\mu_{s_n}, \lambda_2)(p)}{s_n}\cdot \alpha_2(t_2(p)) \mu(t_1(p)) + \sum_p \cos(\mu_{s_n}, \lambda_2)(p)\cdot  \alpha_2(t_2(p))\alpha_1(t_1(p)))
\end{align*}
with error $O({\epsilon_n'}^b)$ with $b\in (0,1]$ and uniform constants depending continuously on the relevant data.
So at scale $\epsilon_n'$, \eqref{eqn:ell_s} is approximated by 
\begin{equation}\label{eqn:R_sum_thm}
\sum_p \left(\frac{  \cos(\mu_{s_n}, \lambda_2)(p)- \cos(\lambda_1, \lambda_2)(p)  }{s_n}\right) \alpha_2(t_2(p))\mu(t_1(p)) + \sum_p \cos(\mu_{s_n}, \lambda_2)(p)\cdot  \alpha_2(t_2(p))\alpha_1(t_1(p)))
\end{equation}
again with error $O({\epsilon_n'}^b)$ with $b\in (0,1]$ and uniform constants.

By the triangle inequality and Proposition \ref{prop:R_sums_prod}, it therefore suffices to prove that the first term in \eqref{eqn:R_sum_thm} goes to zero uniformly as $n$ tends to infinity, and that the second term tends uniformly to $\alpha_2\otimes \alpha_1 \left(\cos_Z(\lambda_1, \lambda_2)\right)$.  However, we will see that the first claim implies the second.

To see uniform convergence  of the first term in \eqref{eqn:R_sum_thm}, we apply Lemma \ref{lem:triple_angle_close} to obtain
\[\left| \cos(\mu_{s_n}, \lambda_2)(p)- \cos(\lambda_1, \lambda_2)(p)  \right| = O\left({\epsilon_n'}\right)= O\left(\epsilon_n^{a} \right) ,\] 
with uniform constants, as usual. There are $O\left(\log^2(1/\epsilon_n')\right)$ intersection points $p$,  $|\alpha_2(t_2(p))| = O(\log(1/\epsilon_n'))$, and $|\mu(t_1(p))|= O(1)$.  By our choice of $s_n$,  the absolute value of the first term in \eqref{eqn:R_sum_thm} goes to zero at rate \[O\left(({\epsilon_n'}\log^3(1/\epsilon_n') \log^2(1/\epsilon_n)\right) = O\left(\epsilon_n^{a}\log^5(1/\epsilon_n) \right) = O\left(\epsilon_n^{b}\right),\]
where $b\in (0, a)$, uniformly near $Z$.

Now, essentially the same argument applied to the second term of \eqref{eqn:R_sum_thm}, also invoking the triangle inequality and Proposition \ref{prop:R_sums_prod} yields  
\begin{align*}
\left|  \sum_{p\in\tau_1(\epsilon_n') \pitchfork \tau_2(\epsilon_n')} \cos(\mu_{s_n}, \lambda_2)(p)\cdot  \alpha_2(t_2(p))\alpha_1(t_1(p))) - \iint_Z\cos(\lambda_1, \lambda_2) ~d\alpha_2 d\alpha_1\right| \\ = O\left(\epsilon_n^{a} \log^4(1/\epsilon_n)\right) = O\left(\epsilon_n^b\right) 
\end{align*}
for $b\in (0, a)$, uniformly near $Z$.  This completes the proof of the theorem.  
\end{proof}

Applying Theorem \ref{thm:cosine} to the results of \S\ref{sec:flow}, and basic properties of the Poisson bracket (Leibniz rule, antisymmetry, and linearity) we immediately obtain the following formula where $\varphi$ as it is in the introduction, i.e., $\varphi$ has maximal invariant measured laminations $[\mu_\pm]$ and simple, positive real spectrum for its linear action on an invariant train track.
A more general formulation can again be obtained from the results in \S\ref{sec:flow}, but this formula is somewhat challenging to parse, even in this simplified setting. 
Recall our notation that if $\alpha_1\in \cH(\lambda_1)$ and $\alpha_2 \in \cH(\lambda_2)$, then 
\[\Cos(\alpha_1, \alpha_2)(Z) = \iint_Z \cos(\lambda_1, \lambda_2)~d\alpha_1d\alpha_2.\]
Recall also that there is a natural duality isomorphism $\ast: \cH(\mu_-) \to \cH(\mu_+)$; to ease notational burden below, let $A_i = \ast\inverse \alpha_i$ and $B_i = \ast\inverse \beta_i$ where $\alpha_1, ..., \alpha_{3g-3}, \beta_1, ..., \beta_{3g-3}$ is a symplectic basis of eigenvectors for the linear $\varphi$ action on $\cH(\mu_+)$.
\begin{corollary}\label{cor:Poisson_pA}
We have an equality 
\begin{align*}
\{-F_{\varphi\inverse}^{\mu_+},F_\varphi^{\mu_-}\}  
= \sum_{i, j = 1}^{3g -g} &\log (\Lambda_i) \log(\Lambda_j)\left( \ell_{\alpha_i}\ell_{A_j} \Cos(\beta_i,B_j)+\right. \\
& \left.\ell_{\alpha_i}\ell_{B_j} \Cos(\beta_i, A_j) + \ell_{\beta_i}\ell_{A_j} \Cos(\alpha_i, B_j)+\ell_{\beta_i}\ell_{B_j} \Cos(\alpha_i,A_j)\right).
\end{align*}
\end{corollary}
\begin{remark}\label{rmk:along_stretch_line}
Along the invariant stretch line $Z_{\mu_+}^t = \sigma_{\mu_+}\inverse (e^t \ast \mu_-) = \sigma_{\mu_+}\inverse (e^t \beta_1)$ for $\mu_+$, many terms vanish:
\begin{align*}
\{-F_{\varphi\inverse}^{\mu_+},F_\varphi^{\mu_-}\} (Z_{\mu_+}^t) =  e^t\log(\Lambda) \sum_{j = 1}^{3g-3} \log(\Lambda_j) \left ( \ell_{A_j}\Cos(\beta_1, B_j) +\ell_{B_j}\Cos(\beta_1, A_j)\right) (Z_{\mu_+}^t).
\end{align*}
This expression can be considered as a measurement for how different the invariant stretch line directed by $\mu_+$ is from the invariant stretch line directed by $\mu_-$.  By $\varphi$ invariance, it admits upper and lower bounds.
\end{remark}

We now give proofs of Corollaries \ref{cor:derivative_intro} and \ref{cor:stretch_variation} from the introduction.
\begin{proof}[Proof of Corollary \ref{cor:derivative_intro}]
We wish to compute $\left(\Sigma_{\lambda_2\lambda_1}\right)_*$, and hence to write \[X_{x_j} = \sum_{i = 1}^{3g-3}a_{i,j}X_{z_i}+c_{i,j}X_{w_i},\] from which will can write the $j$th column in the Jacobian matrix with respect to our chosen symplectic bases.

We can solve for the coefficients $a_{i,j}$ and $c_{i,j}$ by using linearity of the symplectic form, the fact that $\omega_{\WP}(X_{z_i},X_{w_i}) = 1 = -\omega_{\WP}(X_{w_i}, X_{z_i})$, and Lemma \ref{lem:length_dual} to deduce
\[-c_{i,j} = \omega_{\WP} (X_{x_j} , X_{z_i}) = \{\ell_{x_j}, \ell_{z_i}\} \text{ and } a_{i,j} = \omega_{\WP} (X_{x_j} , X_{w_i}) = \{\ell_{x_j}, \ell_{w_i}\}\]
Applying  Theorem \ref{thm:cosine} gives the formulae for $a_{i,j}$ and $c_{i,j}$ in terms of cosines.  The computations for the columns corresponding to $X_{y_j}$ are similar.  

The computation of the inverse matrix is identical, but we can also use the relation \[\begin{pmatrix}
A &B \\
C & D
\end{pmatrix}\inverse = 
\begin{pmatrix}
D^t &-B^t \\
-C^t & A^t
\end{pmatrix}\]
for symplectic matrices written in block form.
\end{proof}
\begin{proof}[Proof of Corollary \ref{cor:stretch_variation}]
The formula follows from the definition of the stretch vector field, bi-linearity of $\omega_{\WP}$, and Theorem \ref{thm:cosine}. 
If $\alpha$ is a transverse measure $\mu$, then $d\log\ell_\mu X_\nu^{\stre}(Z)<1$ with equality if and only if $\mu \le \nu$ \cite[Section 5]{Th_stretch}.  

In fact, if $\alpha \in \cH(\mu)$ is arbitrary, then $d\log\ell_\alpha X_\nu^{\stre}(Z) = 1$, because \[d\log\ell_\alpha X_\nu^{\stre}(Z) = \frac{1}{\ell_\alpha(Z)}\omega_{\WP} (X_\alpha, X_\nu^{\stre}) = \frac{1}{\ell_\alpha(Z)} \omega_{\Th} (\alpha, \sigma_\nu(Z)),\]
while $\omega_{\Th}(\alpha, \sigma_\nu(Z)) = \ell_\alpha (Z)$.  
\end{proof}

\bibliography{ref_pA.bib}{}
\bibliographystyle{amsalpha.bst}
\Addresses

\end{document}